\documentclass[10pt,a4paper]{amsart}%

\usepackage{amsfonts,amsmath,amssymb}
\usepackage{amssymb,amsthm,amsxtra}
\usepackage[usenames]{color}
\usepackage{amscd}
\usepackage{amsthm}
\usepackage{amsfonts}
\usepackage{amssymb}
\usepackage{amsmath}
\usepackage{graphicx}
\usepackage{hyperref}
\usepackage{enumerate}%
\usepackage[all]{xy}
\usepackage[usenames,dvipsnames]{xcolor}
\usepackage{mathrsfs}
\usepackage[margin=1.45in]{geometry}
\usepackage{cleveref}

 \newtheorem*{corollary*}{Corollary}
 \newtheorem*{construction*}{Construction}
 \newtheorem*{definition*}{Definition}
 \newtheorem*{notation*}{Notation}
 \newtheorem*{lemma*}{Lemma}
 \newtheorem*{theorem*}{Theorem}
 \newtheorem*{remark*}{Remark}
 \newtheorem*{example*}{Example}
 \newtheorem*{conjecture*}{Conjecture}
 \newtheorem*{condition*}{Condition}
 \newtheorem*{result*}{Result}
 \newtheorem*{property*}{Property}

 \newtheorem*{cor*}{Corollary}
 \newtheorem*{const*}{Construction}
 \newtheorem*{defn*}{Definition}
 \newtheorem*{notn*}{Notation}
 \newtheorem*{lem*}{Lemma}
 \newtheorem*{thm*}{Theorem}
 \newtheorem*{rem*}{Remark}
 \newtheorem*{exm*}{Example}
 \newtheorem*{conj*}{Conjecture}

 \newtheorem{lemma}{Lemma}[subsection]
 
 \newtheorem{remark}[lemma]{Remark}
 
 \newtheorem{theorem}[lemma]{Theorem}
 \newtheorem{definition}[lemma]{Definition}
 \newtheorem{notation}[lemma]{Notation}

 \newtheorem{corollary}[lemma]{Corollary}

 \newtheorem{thm}[lemma]{Theorem}
 \newtheorem{prop}[lemma]{Proposition}
 \newtheorem{lem}[lemma]{Lemma}
 \newtheorem{defn}[lemma]{Definition}
 \newtheorem{notn}[lemma]{Notation}
 \newtheorem{cor}[lemma]{Corollary}

 \newtheorem{rem}[lemma]{Remark}

 \newtheorem{introtheorem}{Theorem}
 
 \crefname{introtheorem}{theorem}{theorems}
 \Crefname{introtheorem}{Theorem}{Theorems}
  \newtheorem{introthm}[introtheorem]{Theorem}
   \crefname{introthm}{theorem}{theorems}
 \Crefname{introthm}{Theorem}{Theorems}

  \crefname{introcorollary}{corollary}{corollaries}
 \Crefname{introcorollary}{Corollary}{Corollaries}

 \newtheorem{introcor}[introtheorem]{Corollary}
   \crefname{introcor}{corollary}{corollaries}
 \Crefname{introcor}{Corollary}{Corollaries}
 
   \crefname{introconjecture}{conjectures}{conjectures}
 \Crefname{introconjecture}{Conjecture}{Conjectures}
 \newtheorem{introconj}[introtheorem]{Conjecture}
    \crefname{introconj}{conjectures}{conjectures}
 \Crefname{introconj}{Conjecture}{Conjectures}

 \newtheorem{introlem}[introtheorem]{Lemma}
     \crefname{introlem}{lemma}{lemmas}
 \Crefname{introlem}{Lemma}{Lemmas}
 \newtheorem{introremark}[introtheorem]{Remark}
 \crefname{introremark}{remark}{remarks}
 \Crefname{introremark}{Remark}{Remarks}
 \newtheorem{introrem}[introtheorem]{Remark}
  \crefname{introrem}{remark}{remarks}
 \Crefname{introrem}{Remark}{Remarks}

 \newtheorem{introprop}[introtheorem]{Proposition}
   \crefname{introprop}{Proposition}{Propositions}
 \Crefname{introprop}{Proposition}{Propositions}

 \newtheorem{introdefn}[introtheorem]{Definition}
   \crefname{introdefn}{definition}{definitions}
 \Crefname{introdefn}{Definition}{Definitions}
 
   \crefname{intronotn}{notation}{notations}
 \Crefname{intronotn}{Notation}{Notations}
 
   \crefname{introtask}{task}{tasks}
 \Crefname{introtask}{Task}{Tasks}
 
  \crefname{introprob}{problem}{problems}
 \Crefname{introprob}{Problem}{Problems}
 
   \crefname{introquestion}{question}{questions}
 \Crefname{introquestion}{Question}{Questions}

 \crefname{theorem}{theorem}{theorems}
 \Crefname{theorem}{Theorem}{Theorems}
  \crefname{thm}{theorem}{theorems}
 \Crefname{thm}{Theorem}{Theorems}

  \crefname{corollary}{Corollary}{Corollaries}
 \Crefname{corollary}{Corollary}{Corollaries}

   \crefname{cor}{Corollary}{Corollaries}
 \Crefname{cor}{Corollary}{Corollaries}

   \crefname{conjecture}{conjectures}{conjectures}
 \Crefname{conjecture}{Conjecture}{Conjectures}

    \crefname{conj}{conjectures}{conjectures}
 \Crefname{conj}{Conjecture}{Conjectures}

     \crefname{lem}{lemma}{lemmas}
 \Crefname{lem}{Lemma}{Lemmas}

      \crefname{lemma}{Lemma}{Lemmas}
 \Crefname{lemma}{Lemma}{Lemmas}

 \crefname{remark}{remark}{remarks}
 \Crefname{remark}{Remark}{Remarks}

  \crefname{rem}{remark}{remarks}
 \Crefname{rem}{Remark}{Remarks}

   \crefname{rem}{remark}{remarks}
 \Crefname{rem}{Remark}{Remarks}

   \crefname{proposition}{Proposition}{Proposition}
 \Crefname{proposition}{Proposition}{Proposition}

    \crefname{prop}{Proposition}{Propositions}
 \Crefname{prop}{Proposition}{Propositions}

   \crefname{defn}{definition}{definitions}
 \Crefname{defn}{Definition}{Definitions}

   \crefname{notn}{notation}{notations}
 \Crefname{notn}{Notation}{Notations}

   \crefname{task}{task}{tasks}
 \Crefname{task}{Task}{Tasks}

  \crefname{prob}{problem}{problems}
 \Crefname{prob}{Problem}{Problems}

   \crefname{question}{question}{questions}
 \Crefname{question}{Question}{Questions}

\newcommand{\alp}{\alpha}

\newcommand{\Ind}{\operatorname{Ind}}
\newcommand{\ind}{\operatorname{ind}}
\newcommand{\Ext}{\operatorname{Ext}}

\newcommand{\Span}{{\operatorname{Span}}}

\newcommand{\End}{\operatorname{End}}

\newcommand{\GL}{\operatorname{GL}}

\newcommand{\oH}{\operatorname{H}}

\newcommand{\supp}{\operatorname{Supp}}

\newcommand{\Mat}{\operatorname{Mat}}

\newcommand{\Hom}{\operatorname{Hom}}

\newcommand{\bC}{\mathbb{C}}
\newcommand{\bN}{\mathbb{N}}

\newcommand{\bQ}{\mathbb{Q}}
\newcommand{\bZ}{\mathbb{Z}}

\newcommand{\OO}{\mathcal{O}}
\newcommand{\NN}{\mathcal{N}}

\newcommand{\gotG}{\mathfrak{g}}

\newcommand{\Sc}{\cS}

\newcommand{\onto}{\twoheadrightarrow}
\newcommand{\into}{\hookrightarrow}

\providecommand{\fh}{\mathfrak{h}}
\providecommand{\fg}{\mathfrak{g}}
\providecommand{\fm}{\mathfrak{m}}

\providecommand{\fz}{\mathfrak{z}}
\providecommand{\fS}{\mathfrak{S}}

\providecommand{\cB}{\mathcal{B}}

\providecommand{\cE}{\mathcal{E}}
\providecommand{\cF}{\mathcal{F}}

\providecommand{\cH}{\mathcal{H}}
\providecommand{\cI}{\mathcal{I}}
\providecommand{\cJ}{\mathcal{J}}

\providecommand{\cL}{\mathcal{L}}
\providecommand{\cM}{\mathcal{M}}
\providecommand{\cN}{\mathcal{N}}
\providecommand{\cO}{\mathcal{O}}
\providecommand{\cR}{\mathcal{R}}
\providecommand{\cS}{\mathcal{S}}

\providecommand{\cU}{\mathcal{U}}

\providecommand{\g}{\mathfrak{g}}

\newcommand{\simpAr}[2][r]{%
\ar@{}[#1]|-*[@]_{#2}%
}

\newcommand{\Dima}[1]{{{#1}}}
\newcommand{\DimaA}[1]{{{#1}}}
\newcommand{\DimaB}[1]{{{#1}}}
\newcommand{\Rami}[1]{{{#1}}}
\newcommand{\RamiA}[1]{{{#1}}}
\newcommand{\RamiB}[1]{{{#1}}}
\newcommand{\RamiC}[1]{{{#1}}}
\newcommand{\RamiD}[1]{{{#1}}}
\newcommand{\RamiE}[1]{{{#1}}}
\newcommand{\RamiF}[1]{{{#1}}}
\newcommand{\EitanA}[1]{{{#1}}}

\newcommand{\proofend}{\hfill$\Box$\smallskip}
\newcommand{\lbl}[1]{\label{#1}}
\newcommand{\ot}{\leftarrow}

\begin{document}

\title{$\fz$-finite distributions on $p$-adic groups}
\author{Avraham Aizenbud}
\address{Avraham Aizenbud,
Faculty of Mathematics and Computer Science, Weizmann
Institute of Science, POB 26, Rehovot 76100, Israel }
\email{aizenr@gmail.com}
\urladdr{http://www.wisdom.weizmann.ac.il/~aizenr}
\author{Dmitry Gourevitch}
\address{Dmitry Gourevitch, Faculty of Mathematics and Computer Science, Weizmann
Institute of Science, POB 26, Rehovot 76100, Israel }
\email{dimagur@weizmann.ac.il}
\urladdr{http://www.wisdom.weizmann.ac.il/~dimagur}
\author{Eitan Sayag}
\address{Eitan Sayag,
 Department of Mathematics,
Ben Gurion University of the Negev,
P.O.B. 653,
Be'er Sheva 84105,
ISRAEL}
 \email{eitan.sayag@gmail.com}
\address{Alexander Kemarsky,
Mathematics Department, Technion - Israel
Institute of Technology, Haifa, 32000 Israel}
\email{alexkem@tx.technion.ac.il}

%
%
%
%
%
%
%
%
\date{\today}
\maketitle
\begin{abstract}

For a real reductive group $G$,
the center $\fz(\cU(\fg))$ of the universal enveloping algebra of the Lie algebra $\fg$ of $G$ acts on the space of distributions on $G$.
This action proved to be very useful (see e.g. \cite{HCBul,HCReg,Sha,Bar}).

Over non-Archimedean local fields, one can replace the action of  $\fz(\cU(\fg))$  by the action of the Bernstein center $\fz$ of $G$, i.e. the center of the category of smooth representations.
However, this action is not well studied. In this paper we provide some tools to work with this action and prove  the following results.
\begin{itemize}

\item The wavefront set of any $\fz$-finite distribution $\xi$ on $G$ over any point $g\in G$ lies inside the nilpotent cone of $T_g^*G \cong \fg$.

\item Let $H_1,H_2 \subset G$ be symmetric subgroups. Consider the space $\cJ$ of $H_1\times H_2$-invariant distributions on $G$. We prove that the $\fz$-finite distributions in $\cJ$ form a dense subspace. In fact we prove this result in wider generality, where the groups $H_i$ are spherical subgroups of certain type and the invariance condition is replaced by equivariance.
\end{itemize}

Further we apply those results to density and regularity of \DimaB{relative} characters.

 The first  result can be viewed as a version of Howe's expansion of characters.

The second result can be viewed as a spherical space analog of
a classical theorem on density of characters of \DimaB{finite length} representations.
It can also be viewed as a spectral version of Bernstein's localization principle.

In the Archimedean case, the first result is well-known and the second remains open.
\end{abstract}

\tableofcontents

\section{Introduction}\label{sec:intro}


Let $\mathbf{G}$ be a reductive group defined over a non-Archimedean local field $F$. Let $G:=\mathbf{G}(F)$ \EitanA{be the corresponding $l$-group} and let $\Sc(G)$  be the space of locally constant compactly supported functions on $G$. Let $\fz:=\fz(G):=\End_{G \times G}(\Sc(G))$ denote the Bernstein center (see \Cref{subsec:Ber}). The action of $\fz$ on $\cS(G)$ gives rise to the dual action on the space of distributions $\cS^*(G)$.

In this
paper we study $\fz$-finite distributions, i.e. distributions $\xi$ such that \EitanA{$\dim (\fz \cdot \xi)< \infty $}.

\subsection{Wavefront set of $\fz$-finite distributions}
Our first result concerns the  wavefront set of such distributions.
For  $x\in G$ let $WF_{x}(\xi)$ denote the intersection of the wavefront set of $\xi$ with the cotangent space $T_x^*G$ (see \Cref{subsec:WF}).

In \Cref{sec:WF} we prove
\begin{introthm}\label{thm:main}
Suppose that $F$ has characteristic zero. Let $\xi \in \cS^*(G)$ be a $\fz$-finite distribution. Then for any $x\in G$ we have
\begin{equation}
WF_{x}(\xi) \subset \cN
\end{equation}
where $\cN \subset \fg^*$ is the nilpotent cone, and we identify \DimaA{the Lie algebra} $\fg$ with $T_xG$ using the right action\footnote{Since $\cN$ is invariant by conjugation it does not matter whether we use the right or the left action.}.
\end{introthm}

Our main tool is the theory of \DimaB{special} balls. This theory was developed for $\mathbf{G}= \GL_n$ in \cite{S} \DimaB{(where these balls were called fuzzy balls)}, using some ideas from \cite{H,H2}.  In \Cref{subsec:PrelFuzzy,sec:balls} we recall the relevant part of this theory and adapt it to general reductive groups.

\DimaA{
\begin{introremark}
We need the characteristic zero assumption since we use the exponentiation map in order to identify a neighborhood of zero in $\fg$ with a neighborhood of the unit element in $G$. For $G=\GL_n$ one can use the map $X \mapsto \mathrm{Id} + X$ (as in \cite{S}) and drop the assumption on the characteristic. \RamiE{It is likely that for other classical groups one can use the Cayley map, and considerably weaken the assumption on the characteristic. The general case can be possibly treated using \cite[Appendix A]{AS}.}
\end{introremark}
}
\subsection{Density of $\fz$-finite distributions}
\DimaB{
The next results of this paper depend on closed subgroups of $G$.
We will require some conditions on these subgroups. We will describe those conditions in \Cref{defn:FinMult} below. If a subgroup $H\subset G$
satisfies these conditions we will call the pair \Rami{$(G,H)$} \textit{a \Rami{pair of finite type}}.    Conjecturally, if $F$ has characteristic zero then this holds for all spherical pairs. As explained below in \Cref{subsubsec:FinGen}, many cases of this conjecture follow from \Cref{sec:FinGen} and \cite[Theorem A]{AAG}, based on \cite[Theorem 5.1.5]{SV}, and \cite{Del}. Those cases include all symmetric pairs of reductive groups. 

\begin{introthm}[see \Cref{sec:Dense} below] \label{thm:Dense}
\Rami{Let $H_1,H_2\subset G$ be two (closed) subgroups and $\chi_i$ be characters of $H_i$.}
Consider the two-sided action of $H_1\times H_2$ on $G$ and let $$\cI:=\Sc^*(G)^{(H_1\times H_2,\chi_1\times\chi_2)}$$ be the space of $(H_1\times H_2,\chi_1\times\chi_2)$-equivariant distributions on $G$. Note that the Bernstein center $\fz$ acts on $\cI$.
Assume that  the \Rami{pairs $(G,H_i)$ are of finite type}.

 Then the space of $\fz$-finite distributions in $\cI$ is dense in $\cI$.
\end{introthm}
}
\subsection{Applications}
\DimaB{In this subsection we continue to work in the notation and assumptions of \Cref{thm:Dense}.}

Important examples of $\fz$-finite distributions in $\cI$ are $(H_{1}\times H_{2},\chi_1\times \chi_2)$-\DimaB{relative} characters of  \DimaB{finite length}  representations (see \Cref{def:SphChar}). It turns out that those examples are exhaustive. Namely, we have the following proposition.

\begin{introprop}[see \Cref{sec:DistChar} below]\label{prop:DistChar}
Any $\fz$-finite distribution in $\cI$ is an  $(H_1\times H_2, \chi_1\times \chi_2)$-\DimaB{relative} character of some \DimaB{finite length} representation of $G$.
\end{introprop}

Together with \Cref{thm:Dense} it implies

\begin{introcor}\label{cor:SphCharDense}
The space of $(H_1\times H_2, \chi_1\times \chi_2)$-\DimaB{relative} characters of \DimaB{finite length} representations of $G$ is dense in $\cI$.
\end{introcor}

\Cref{thm:main} provides a simple proof of the easy part of Harish-Chandra's regularity theorem \cite{HC_sub,H2}, namely the regularity of the character on the set of regular semi-simple elements. In \Cref{sec:reg} we generalize this result to the realm of spherical pairs. For that, we introduce the notion of $H_1\times H_2$-cotoric elements and prove
the following result.
\begin{introcor}\label{thm:reg}
Suppose that $F$ has characteristic zero, \DimaB{and $H_i$ are $F$-points of algebraic groups $\mathbf{H}_i\subset \mathbf{G}$.} Let $\xi \in \cI$ be a $\fz$-finite distribution.
Then $\xi$ is smooth in the neighborhood of any $H_1\times H_2$-cotoric element.
\end{introcor}
This results generalizes the main result of \cite[\S 5]{RR}, since if $\EitanA{H:=}H_1=H_2$ is a symmetric subgroup then the $H$-regular semisimple elements are cotoric (see \Cref{lem:SymSmooth}).


\EitanA{Combining \Cref{thm:main,thm:Dense} we obtain the following tool to study invariant distributions}:

\begin{introcor}\label{cor:WF}
The subspace of distributions in $\cI$ whose wavefront set at any point is contained in the nilpotent cone in the dual Lie algebra $\fg^*$ is dense in $\cI$.
\end{introcor}

\subsection{Related results}

The germ at the unit element of the character of an irreducible representation of $G$
\EitanA{can be presented} as a linear combination of
Fourier transforms of invariant measures of nilpotent orbits. This was shown in \cite{H} for $\mathbf{G}=\GL_n$ and in \cite{HCQ} for general $\mathbf{G}$. This cannot be naively generalized to the case of symmetric pairs, since the nilpotent orbital integrals are not defined for symmetric spaces in general. However, in \cite[\S 7]{RR} it is shown that the germ at the unit element of a \DimaB{relative} character is a Fourier transform of a distribution supported on the nilpotent cone.

\Cref{thm:main}  can be viewed as a version of these results, which gives less information but works in  wider generality. Namely, it
implies that the germ of \DimaB{any relative character of any finite length representation} is a Fourier transform of a distribution supported near the nilpotent cone.

Distributions \EitanA{arising in} representation theory are often $\fz$-finite. In the Archimedean case (where $\fz$ means the center of the universal enveloping algebra of the Lie algebra) this was widely exploited. For example it was used to prove the Harish-Chandra regularity theorem (\cite{HCBul,HCReg}),
uniqueness of Whittaker models (\cite{Sha}) and Kirillov's conjecture
(\cite{Bar}). Recently, it was used in \cite{JSZ} to prove uniqueness of Ginzburg-Rallis models and in \cite{AG_ZR} to show non-vanishing of Bessel-like functions. However, in the non-Archimedean case there were no tools that use finiteness of distributions under the Bernstein center. This work provides such a tool.

A classical result (see \cite[\S A.2]{DKV} and \cite[Appendix]{Kaz})
 says that characters of \DimaB{finite length} representations span a dense subspace of the space of conjugation-invariant distributions on $G$. One can view \Cref{cor:SphCharDense} as the relative counterpart of this result.

One can attempt to generalize \Cref{thm:Dense} in the following direction. Let an $l$-group $G$ act on an $l$-space $X$, and let $\cE$ be a $G$-equivariant sheaf on $X$. Let a complex commutative algebra $A$ act on $\cS(X,\cE)$. Let $\cI:=\cS^*(X,\cE)^G$ be the space of invariant distributional sections of $\cE$. Assume that $A$ preserves $\cI$. Is the space of $A$-finite distributions in $\cI$ dense in $\cI$? Another important special case of this question is the case when $A=\cS(Y)$ where $Y$ is some $l$-space and the action of $A$ on $\cS(X)$ is given by a map from $X$ to $Y$ \DimaB{and $G$ acts on the fibers of this map}.
In this case the positive answer is given by Bernstein's localization principle \cite[\S 1.4]{BerP}.
Thus, one can view \Cref{thm:Dense} as a spectral version of Bernstein's localization principle.

The Archimedean analogs of \Cref{thm:Dense} as well as of Bernstein's localization principle are not known in general.

\subsection{Tools developed in this paper}

\subsubsection{Pairs of finite type}\label{subsubsec:FinGen}
\Rami{
\begin{introdefn}\label{defn:FinMult}

Let $H<G$ be a closed subgroup and $\chi$ be its character.
We say that the pair $(G,H)$ has \emph{finite type} if
for any character $\chi$ of $H$ and any compact open 
subgroup $K<G$, the module $(\ind_{H}^{G}\chi)^K$
  over the Hecke algebra
$\mathcal{H}_K(G)$ is
finitely generated. \DimaB{Here, $\ind_{H}^{G}\chi$ denotes the compact induction.}
\end{introdefn}

In \Cref{sec:FinGen} we give  the following criterion for  pairs to be of  finite type.

\begin{introtheorem}[cf. \Cref{thm:ExpFinGen}]\label{thm:FinGen}
Let $H$ be a closed subgroup of $G$. Let $\mathbf{P}$ be a minimal parabolic  subgroup of $\mathbf{G}$ and $P=\mathbf{P}(F)$. Suppose that $H$ has finitely many orbits on $G/P$.
Suppose that for any
irreducible smooth representation $\rho$ of $G$ and any character $\chi$ of $H$ we have
\begin{equation} \label{eq:FinMult}
\dim\Hom_{H}(\rho|_{H}, \chi) < \infty .
\end{equation}
 Then the  pair  $(G,H)$ is of  finite type.
\end{introtheorem}

\begin{introrem}$\,$
\begin{enumerate}
\item In fact, \Cref{thm:ExpFinGen} gives a more precise statement, which deduces finite generation of $(\ind_{H}^{G}\chi)^K$ from formula \eqref{eq:FinMult} for a specific character derived from $\chi$. One can strengthen other results of this paper in a similar way. However, this will require longer bookkeeping that we chose to avoid.

\item An incomplete version of \Cref{thm:ExpFinGen} appeared in \cite{AAG}, see \Cref{rem:AAG} for more details.

\item The condition \eqref{eq:FinMult} is proven in \cite{Del} and \cite[Theorem 5.1.5]{SV} for many cases, including arbitrary symmetric pairs  over a field with characteristic different from 2.
\end{enumerate}
\end{introrem}}
\subsubsection{Representations generated by $\fz$-finite distributions}
In order to prove \Cref{prop:DistChar} we prove the following lemma:
\begin{introlem}[see \Cref{sec:DistChar}]\label{lem:adm}
\Rami{Let $(G,H)$ be a  pair of   finite type. Let $\chi$ be a character of $H$.
Consider the left action of $H$ on $G$ and let}
$\xi \in \cS^*(G)^{(H,\chi)}$ be an $(H,\chi)$-equivariant $\DimaA{\fz}$-finite distribution. Then \RamiB{both $\cS(G)*\xi$ and $\xi*\cS(G)$} are \DimaB{finite length} representations of $G$.
\end{introlem}

This lemma implies the following corollary:

\begin{introcor}\label{cor:adm}
Let $\xi \in \cS^*(G)$ be a $\fz$-finite distribution. Then $\RamiB{\cS(G)*\xi * \cS(G)}$ is a \DimaB{finite length} representation of $G\times G$.
\end{introcor}

\subsubsection{Fuzzy balls}
The theory of \DimaB{special} balls was developed in \cite{S} based on \cite{H,H2} for $\mathbf{G}=GL_n$. This theory implies that any irreducible representation is annihilated by a certain collection of elements of the Hecke algebra. In \Cref{subsec:PrelFuzzy,sec:balls} we adapt this statement to representations of a general reductive group (see \Cref{thm:FuzzyFin}).

\subsubsection{Relations between convolution and exponentiation}
The exponentiation maps an open neighborhood $U$ of zero in the Lie algebra of $G$ to $G$. This gives rise to a map of the algebra $\cS(U)$ of smooth compactly supported functions on $U$ (with respect to convolution) to the Hecke algebra of $G$. Unfortunately, this map is not a homomorphism. In \Cref{prop:FouExp} we show that it does behave as a homomorphism on certain pairs of functions.

\subsection{Idea of the proof}

\subsubsection{Sketch of the proof of  \Cref{thm:main}}

We first analyze the representation generated by $\xi$ under the two-sided action of the Hecke algebra  $\cH(G)$, which has \DimaB{finite length} by \Cref{cor:adm}.
Then we use the theory of \DimaB{special} balls, that produces, for any \DimaB{finite length} representation, a large collection of elements   in the Hecke algebra $\cH(G)$ that annihilate it. Those elements will also annihilate $\xi$. In other words, for certain   $e_B \in \cH(G)$ we have the following vanishing of convolutions
\begin{equation}\label{=:van}
e_B * \xi=0
\end{equation}
Next we want to linearize this information. For this we use the exponentiation map and \Cref{prop:FouExp}. Unfortunately,  \Cref{prop:FouExp} is not directly applicable to the pair $(e_B,\xi)$. However, we use the vanishing \eqref{=:van}
to construct other vanishing convolutions, to which \Cref{prop:FouExp} is applicable. Thus we get that  certain convolutions on the Lie algebra vanish. Those vanishings imply  the desired restriction on the wave front set.

\subsubsection{Sketch of the proof of \Cref{thm:Dense}}
Let us assume for simplicity that $\chi_i$ are trivial and $H_i$ are unimodular.
To prove \Cref{thm:Dense} we first note that $\cI$ is dual to the space $\cS(G)_{H_1\times H_2}$ of  $(H_1\times H_2)$-coinvariants of $\Sc(G)$. We can decompose $\Sc(G) $ to a direct sum with respect to Bernstein blocks. This leads to a decomposition of $\cS(G)_{H_1\times H_2}$. The finite type assumption implies that each summand is finitely generated over $\DimaA{\fz}$. Thus Artin-Rees Lemma and Hilbert's Nullstellensatz imply that the space of $\DimaA{\fz}$-finite functionals on those summands is dense in the space of arbitrary functionals.

For technical reasons, it is more convenient to work with unions of Bernstein blocks which correspond to compact open subgroups of $G$ than with individual Bernstein blocks.


\subsection{Future applications}

We believe that \Cref{cor:WF}  can be used in order to prove the following analog of Harish-Chandra's density theorem \cite[Theorem 3.1]{HCQ}.

\begin{introconj}
Suppose that $G$ is quasisplit. Let $B$ be a Borel subgroup of $G$, $U$ be the nilradical of $B$, $\psi$ be a non-degenerate character of $U$,  $H\subset G$ be a reductive spherical subgroup and $X=G/H$. Let $\cO$ be the union of all open $B$-orbits in $X$.

Then the sum of the  one-dimensional spaces $\Sc^*(Ux)^{U,\psi}$, where $x$ ranges over $\cO$, is dense in $\Sc^*(X)^{U ,\psi }$.
\end{introconj}

In \DimaB{the subsequent} paper \cite{AGK} we prove a non-Archimedean analog of \cite{AG_ZR}, which we consider as a step towards this conjecture. Namely, we use \Cref{thm:main} in order to prove that under certain conditions on $H$  any $\fz$-finite distribution $\xi \in \Sc^*(X)^{U,\psi }$ which is supported in the complement to $\cO$  vanishes.

In \DimaB{the subsequent} work  \cite{AGM} we prove that the set of cotoric elements is open and dense in $G$ if $H_1,H_2$ are spherical subgroups. By
\Cref{thm:reg} this implies that $H_1\times H_2$-  \DimaB{relative} characters are smooth almost everywhere.
In fact, in \cite{AGM} we show that the dimension of the variety
$$ \fS=\EitanA{\{(g,\alpha) \in G \times \g^*\, \mid\, \alpha \text{ is nilpotent, } \langle \alpha, \fh_1 \rangle =0, \ \langle \alpha, Ad^*(g)(\fh_2) \rangle =0\}
}\subset T^*G$$
equals the dimension of $G$. \Cref{thm:main} implies that the wavefront set of any $H_1\times H_2$-  \DimaB{relative} character lies in $\fS$. Thus we obtain a certain version of holonomicity for \DimaB{relative} characters.


\subsection{Structure of the paper}

In \Cref{sec:prel} we give the necessary preliminaries
on the Bernstein center.

In \Cref{sec:DistChar} we prove \Cref{lem:adm} and deduce \Cref{prop:DistChar}.

In \Cref{sec:Dense} we prove \Cref{thm:Dense}.

In \Cref{sec:WF} we prove \Cref{thm:main}. In \Cref{subsec:WF,subsec:PrelFuzzy} we give the necessary preliminaries on  wavefront set and on \DimaB{special} balls. In \Cref{sec:PfWF} we deduce \Cref{thm:main} from two main ingredients, which we prove  in \Cref{subsec:PfEigenFinSupp,subsec:pfFouExp}.
In \Cref{subsec:PfEigenFinSupp} we prove the vanishing \eqref{=:van}. In \Cref{subsec:pfFouExp} we prove \Cref{prop:FouExp} that states that exponentiation commutes with convolution in certain cases.
In \Cref{sec:reg} we prove  \Cref{thm:reg} and \Cref{lem:SymSmooth}, which allows to specialize \Cref{thm:reg} to the  symmetric pair case and thus obtain a generalization of \cite[\S 5]{RR}.

In \Cref{sec:balls} we prove the statements on \DimaB{special} balls that were formulated without proof in \Cref{subsec:PrelFuzzy}.

In \Cref{sec:FinGen} we prove \Cref{thm:FinGen}.

\subsection{Acknowledgements}
\DimaA{We thank Moshe Baruch for motivating questions, Nir Avni and Erez Lapid for fruitful discussions, Joseph Bernstein  for useful remarks and Yuval Flicker for careful proofreading.}

A.A. was partially supported by NSF grant DMS-1100943, ISF grant 687/13, \DimaB{BSF grant 2012247, and a Minerva foundation grant.};

D.G. was partially supported by ISF grant 756/12,  \DimaB{ERC StG grant 637912}, and a Minerva foundation grant.

E.S. was partially supported by ISF grant 1138/10, and  ERC grant 291612.

\section{Preliminaries}\label{sec:prel}

\subsection{Conventions}

The following conventions will be used throughout the paper.

\begin{itemize}
\item Fix a non-Archimedean local field $F$.
\item All the algebraic groups and algebraic
varieties that we consider  are defined over $F$.
We will use bold letters, e.g. $\mathbf{G},\mathbf{X}$ to denote algebraic groups and varieties defined over $F$, and their non-bold versions to denote the $F$-points of these varieties, considered as $l$-spaces or $F$-analytic manifolds (in the sense of \cite{Ser}).

 \item We will use capital Latin letters to denote $F$-analytic groups and algebraic groups, and the corresponding Gothic letters to denote their Lie algebras.

\item For an $l$-group $H$
\Rami{
\begin{itemize}
\item let  $\cM(H)$ denote the category of smooth representations of $H.$

\item Let $\Delta_H$ denote the modular character of $H$, \DimaB{i.e. the quotient of the right Haar measure by the left one.}

\item If $H$ acts on an $l$-space $X$ and  $x\in X$, we denote by $H_x$ the stabilizer of $x$.

\item If  $V$ is a representation of $H$ we denote by $V_H$ the space of coinvariants
$$V_H:=V/(\Span\{v-hv\, | \, v\in V, \, h\in H\}).$$
\end{itemize}
}
\item Fix  a reductive group $\mathbf{G}$.
\item
\RamiB{ For a sheaf $\cF$ \DimaB{of $\bC$-vector spaces }on an $l$-space $X$ we denote by $\Sc(X,\cF)$ the space of compactly-supported sections and by $\Sc^*(X,\cF)$ the dual space.}

\item \DimaB{For a  compact open subgroup $K<G$ we denote by $\cH_K(G)$ the corresponding Hecke algebra.}

\end{itemize}

\subsection{Bernstein center}\label{subsec:Ber}
%
%

In this subsection we review the basics of the theory of the Bernstein center from \cite{BD}.
\begin{defn}
The \emph{Bernstein center} $\DimaA{\fz}:=\fz(G)$ is the algebra of endomorphisms of the identity functor of the category $\cM(G)$ of smooth representations of $G$.
\end{defn}

\begin{defn}\label{def:split}
Let  $K<G$ be a compact open subgroup.
\RamiB{For $V\in \cM(G)$ denote by $V^{(K)}$ the subrepresentation generated by its $K$-fixed vectors. Denote also $$\cM_K(G):= \{V \in \cM(G)\, | V \text{ = }V^{(K)}\}$$} and $$\cM_{K} (G)^{\bot}:=
\{V \in \cM(G)\, | V^K = 0\}.$$

We have a functor $\mathcal{P}_K( V):=V^K$ from $\cM(G)$ to the category $\cM(\cH_K(G))$ of modules over $\cH_K(G)$.
We call $K$ a \emph{splitting subgroup} if the category $\cM(G)$ is the direct sum of the categories $\cM_{K}(G)$ and $\cM_{K}(G)^{\bot}$, and the functor $\mathcal{P}_K:\cM_{K}(G) \to \cM(\cH_K(G))$ is an equivalence of categories. \end{defn}
\EitanA{
\begin{remark}
Recall that an abelian category $\mathcal{A}$ is a direct sum of two abelian subcategories $\mathcal{B}$ and $\mathcal{C}$, if every object of $\mathcal{A}$ is isomorphic to a direct sum of an object in $\mathcal{B}$ and an object in $\mathcal{C}$, and, furthermore, that there are no non-trivial
morphisms
between objects of $\mathcal{B}$ and $\mathcal{C}$.
\end{remark}
}

\begin{thm}[\cite{BD}]\label{thm:Ber}
$\,$
\begin{enumerate}
\item The center of the algebra $\End_G(\cS(G))$ of $G$-endomorphisms of $\cS(G)$  is the algebra $\End_{G\times G}(\cS(G)) $ and the natural morphism from $\DimaA{\fz}$ to this center is an isomorphism.

\item \label{it:split} The set of splitting subgroups defines a basis at 1 for the topology of $G$.

\item For any splitting open compact subgroup $K\subset G$ we have

\begin{enumerate}
\item \label{1}   The center $\fz(\cH_K(G))$ of the $K$-Hecke algebra is a direct summand of the Bernstein center $\DimaA{\fz}$. In particular, the natural map $\DimaA{\fz}\to \fz(\cH_K(G))$ is onto.
\item \label{2} \DimaA{The algebra} $\cH_K(G)$ is finitely generated \RamiB{as a module} over its center $\fz(\cH_K(G))$, and thus also over $\DimaA{\fz}$.
\item \label{3} The algebra $\fz(\cH_K(G))$ is finitely generated over $\bC$ and has no nilpotents.
\end{enumerate}
\end{enumerate}
\end{thm}

\section{$\fz$-finite distributions and \DimaB{relative} characters}\label{sec:DistChar}


%
%

\subsection{\DimaB{Finite length} representations, $\fz$-finite distributions and proof of \Cref{lem:adm}}






\RamiB{We start with several criteria for admissibility of smooth representations. For these criteria we will need the following definition.

\begin{defn}
We say that a smooth representation $\pi$ of $G$ is
\begin{itemize}
\item  \emph{locally finitely generated} if for any compact open subgroup $K\subset G$ the module $\pi^K$ is  finitely generated over the Hecke algebra $\cH_K(G)$,
\item \emph{$\fz$-finite} if there exists an ideal $I \subset \fz$ of finite codimension that acts by zero on $\pi$.
\end{itemize}
\end{defn}

\begin{lem}\label{lem:ZfinAdm}
Let $\pi \in \cM(G)$ be a $\fz$-finite smooth representation. Assume that for any compact open subgroup $K\subset G$ the space $\pi^K$ is finite-dimensional.
\EitanA{Then $\pi$ has finite length.}
\end{lem}
\begin{proof}

It is enough to show that $\pi\subset \cM_K(G)$ for some splitting subgroup $K\subset G$. Let  $I \subset \fz$ be an ideal of finite codimension that acts by zero on $\pi$.
For any splitting $K$ denote by $i_K \subset \fz$ the idempotent that acts by identity on $\cM_K(G)$ and by zero on $\cM_{K} (G)^{\bot}$. Let $j_K$ be the image of $i_K$ in $\fz/I$. Since $\fz/I$ is finite-dimensional there exists a splitting $K$ such that $j_K=j_{K'}$ for any splitting subgroup $K'\subset K$,  thus $\pi^{K'}\subset\pi^{(K')}=\pi^{(K)}$ and thus, by \Cref{thm:Ber}\eqref{it:split}, $\pi=\pi^{(K)}$.
\end{proof}

\begin{cor}\label{cor:zLocFinGenAdm}
Any $\fz$-finite locally finitely generated $\pi \in \cM(G)$ \DimaB{has finite length}.
\end{cor}
\begin{proof}
By \Cref{lem:ZfinAdm} and \Cref{thm:Ber}\eqref{it:split} it is enough to show that $\pi^K$ is finite-dimensional for any  splitting subgroup $K\subset G$. This follows from \Cref{thm:Ber}(\ref{1},\ref{2}).
\end{proof}

\begin{prop}\label{prop:LocFinAdm}
Let $\pi \in \cM(G)$ be  locally finitely generated. Then
\begin{enumerate}[(i)]
\item \label{it:AdmQuo}any $\fz$-finite quotient $\rho$ of $\pi$ \DimaB{has finite length},
\item \label{it:AdmSub}any $\fz$-finite subrepresentation $\rho$ of $\widetilde{\pi}$ is \DimaB{has finite length}.
\end{enumerate}
\end{prop}

\begin{proof}
Part \eqref{it:AdmQuo} follows from \Cref{cor:zLocFinGenAdm}. To prove part \eqref{it:AdmSub} denote by $\rho_{\bot} \subset \pi$ the joint kernel of all the functionals in $\rho$. Then $\rho \subset (\rho_{\bot})^{\bot}\cong\widetilde{\pi/\rho_{\bot}}$. Since $\pi/\rho_{\bot}$ \DimaB{has finite length} by part \eqref{it:AdmQuo}, we get that $\rho$ \DimaB{has finite length}.
\end{proof}

\begin{proof}[Proof of \Cref{lem:adm}]
$\,$
\begin{enumerate}[(i)]
\item \label{it:VKFin} \emph{Proof that $\Sc(G)*\xi$ \DimaB{has finite length}.}\\
Consider the natural epimorphism $\Sc(G)\onto \Sc(G)*\xi$. It is easy to see that there exists a character $\chi'$ of $H$ such that this epimorphism factors through $\ind_H^G(\chi')$. Since $(G,H)$ has finite type, $\ind_H^G(\chi')$ is locally finitely generated and thus, by \Cref{prop:LocFinAdm}\eqref{it:AdmQuo},
$\Sc(G)*\xi$ \DimaB{has finite length}.

%

\item \emph{Proof that  $\xi*\Sc(G)$ \DimaB{has finite length}.}\\
Let $G$ act on itself by $g \cdot x = xg^{-1}$. This gives rise to an action of $G$ on $\Sc^*(G)^{(H,\chi)}$. Let $\cF$ be the natural equivariant sheaf on $X= G/H$ such that $\Sc^*(G)^{(H,\chi)}\cong\Sc^*(X,\cF)$. Consider $\xi$ as an element in $\Sc^*(X,\cF)$. Then
$$\xi*\Sc(G) \into \widetilde{\Sc(X,\cF)}=\widetilde{\ind_H^G(\chi'')}$$
for some character $\chi''$ of $H$, and \Cref{prop:LocFinAdm}\eqref{it:AdmSub} implies  that
$\xi*\Sc(G)$ \DimaB{has finite length}.
\end{enumerate}
\end{proof}
}
%
\subsection{\DimaB{Relative} characters and proof of \Cref{prop:DistChar}}$\,$

Let us recall the definition of  \DimaB{relative} character.
\begin{defn}\label{def:SphChar}
Let $(\pi,V)$ be a \DimaB{finite length} representation of $G$. Let $(\tilde \pi,\tilde V) $ be its smooth contragredient. Let $H_1,H_2 \subset G$ be subgroups and $\chi_1,\chi_2$ be their characters. Let $l_{1}\in(V^*)^{H_1,\chi_1^{-1}}$ and $l_{2}\in(\tilde V^*)^{H_2,\chi_2^{-1}}$ be equivariant functionals. The \emph{\DimaB{relative} character} $\Xi^{\pi}_{l_1,l_2} \in \cH(G)^*$ is the generalized function on $G$ given by
\begin{equation}
\Xi^{\pi}_{l_1,l_2}(f) := \langle l_2 , \pi^*(f) l_1 \rangle.
\end{equation}

We refer to such \DimaB{relative} characters as  $(H_1\times H_2, \chi_1\times \chi_2)$-\DimaB{relative} characters of $\pi$.
\end{defn}
Since we can identify $\cI$ with the space  $(\cH(G)^*)^{H_{1}\times H_{2},\chi_1\times \chi_2}$ of invariant generalized functions, we can consider the \DimaB{relative} character as an element in $\cI$.

\begin{lem}[see \Cref{subsec:FinDual}]\label{lem:admFin}
Let $(G,H)$ be a  pair of  finite type. Let $V$ be a \DimaB{finite length} representation of $G$ and $\chi$ be a character of $H$. Then $\dim V_{(H,\chi)}< \infty$.
\end{lem}

\begin{proof}[Proof of \Cref{prop:DistChar}]
Let $\xi\in \cI$.  Consider the pullback of $\xi$ to $G\times G$ under the multiplication map. This gives us a $G$-invariant bilinear form $B$ on $\cH(G)$. Let $\cL$ be its left kernel and $\cR$ be its right kernel, $M:=\cL\backslash  \cH(G)$ and $N:=\cH(G)/\cR$. \Rami{We consider the right $G$-module $M$ as a left one using the anti-involution $g \mapsto g^{-1}$.  We get} a non-degenerate pairing between $M$ and $N$. \Cref{lem:adm} implies that $M$ and $N$ \DimaB{have finite length} and thus $M=\tilde N$.  We can consider the form $B$ as an element in $(M_{H_1,\Rami{\chi_1}}\otimes N_{H_2,\Rami{\chi_2}})^*$.
\Rami{Since the pairs $(G,H_i)$ are of   finite type}, \Cref{lem:admFin} implies that  $M_{H_1,\Rami{\chi_1}}$ and $N_{H_2,\Rami{\chi_2}}$ are finite-dimensional and thus
$$(M_{H_1,\chi_1}\otimes N_{H_2,\chi_2})^*\cong (M_{H_1,\chi_1})^*\otimes (N_{H_2,\chi_2})^*\cong (M^*)^{H_1,\chi_1^{\Rami{-1}}}\otimes (N^*)^{H_2,\chi_2^{\Rami{-1}}}.$$
Therefore $B$ defines an element in $(M^*)^{H_1,\chi_1^{\Rami{-1}}}\otimes (N^*)^{H_2,\chi_2^{\Rami{-1}}}$ which can be written in the form $B=\sum_{i=1}^k l_1^i \otimes l_2^i$.
Let $$l_1:=(l_1^1,\dots,l_1^k)\in ((M^k)^*)^{H_1,\chi_1^{\Rami{-1}}},\,l_2:=(l_2^1,\dots,l_2^k)\in ((N^k)^*)^{H_2,\chi_2^{\Rami{-1}}}.$$
 It is easy to see that $$\xi=\sum_{i=1}^k \Xi^M_{l_1^i,l_2^i}=\Xi^{M^k}_{l_1,l_2}.$$
\end{proof}

\section{Density of $\fz$-finite distributions}\label{sec:Dense}
\setcounter{lemma}{0}
For the proof of \Cref{thm:Dense} we will need the following 
\RamiB{lemma}.

\Dima{
\begin{lem}\label{lem:(K)}
Let $H<G$ be a closed subgroup and $\chi$ be a character of $H$.
Then there exists a character $\chi'$ of $H$ such that for any
 $V\in \cM(G)$ and any splitting subgroup $K\subset G$ \DimaB{(see \Cref{def:split})} we have
$$(V^{(K)})_{(H,\chi)}\cong (\ind_H^G\chi')^K \otimes _{\cH_K(G)} V^K .$$
Here we consider the left ${\cH_K(G)}$-module $(\ind_{H}^G\chi')^K $ as a right one using the anti-involution $g \mapsto g^{-1}.$\end{lem}
\begin{proof}
First note that $V^{(K)} \cong \cH(G) \otimes _{\cH_K(G)} V^K,$ where the action of $G$ is given by the left action on $\cH(G)$.
\RamiB{
Let $H$ act on $G$ from the left and $G$ act on itself by $g\cdot x = x g^{-1}$. This gives an action of $G$ on $\cH(G)_{(H,\chi)}$. It is easy to see that we have an isomorphism  $\cH(G)_{(H,\chi)}\cong (\ind_H^G\chi')$ for some character $\chi'$ of $H$.
 Now}
$$(V^{(K)})_{(H,\chi)}\cong \cH(G)_{(H,\chi)} \otimes _{\cH_K(G)} V^K \cong  \ind_H^G\chi' \otimes _{\cH_K(G)} V^K \cong  (\ind_H^G\chi')^K \otimes _{\cH_K(G)} V^K.
$$
\end{proof}
}
\begin{lem}\label{lem:comalg}
Let $A$ be a unital commutative algebra finitely generated over $\bC$. Let $M$ be a finitely generated $A$-module, and $M^*$ denote the space of all $\bC$-linear functionals on $M$. Then the space of $A$-finite vectors in $M^*$ is dense in $M^*$.
\end{lem}
\begin{proof}
It is enough to show that the intersection of the kernels of all $A$-finite functionals on $M$ is zero. Let $v$ be an element of this intersection,  $\fm\subset A$ be any maximal ideal and $i$ be any integer. Then $M/ \fm^iM$ is finite-dimensional over $\bC$ and thus any functional on it defines an $A$-finite functional on $M$. \EitanA{Such a}
 functional vanishes on $v$, and thus the image of $v$ in $M/ \fm^iM$ is zero. We conclude that $v$ belongs to the space $\bigcap_{\fm} \bigcap_{i} (\fm^iM)$, which is zero by the Artin-Rees lemma.
\end{proof}
\Dima{
\begin{proof}[Proof of \Cref{{thm:Dense}}]
Denote $X_i:=G/H_i$. For some line bundle $\cF_1$ on $X_1$ we have
$$\cI \cong \Sc^*(X_1,\cF_1)^{(H_2,\chi_2)} \cong (\Sc(X_1,\cF_1)_{(H_2,\chi_2^{-1})})^*.$$
Thus it is enough to show that for any $f\in \Sc(X_1,\cF_1)_{(H_2,\chi_2^{-1})}$ such that $\langle \xi, f\rangle=0$ for any $\DimaA{\fz}$-finite distribution $\xi \in \cI,$ we have $f=0$. Let $f$ be like that. Let $K<G$ be a splitting open compact subgroup that fixes a representative of $f$ in $\Sc(X_1,\cF_1)$. Note that $V:=\Sc(X_1,\cF_1)=\ind_{H_{1}}^G\chi_1'$ for some character $\chi_1'$ of $H_1$ . \RamiB{Since $K$ is a splitting subgroup, $V^{(K)}$  is a direct summand of $V$ as a $G$-representation. Hence} $M:=(V^{(K)})_{(H_2,\chi_2^{-1})}$ is a direct summand of $V_{(H_2,\chi_2^{-1})}$ as a $\DimaA{\fz}$-module which contains $f$. Therefore it is enough to show that the space of $\DimaA{\fz}$-finite vectors in $M^*$ (which by \Cref{thm:Ber}\eqref{1} equals  the space of $\fz(\cH_K(G))$-finite vectors in $M^*$) is dense in $M^*$. By \Cref{lem:(K)}, \RamiC{there exists a character $\chi_2'$ of $H_2$ such that}
$$
M=(\ind_{H_2}^G\chi_{2}')^K \otimes _{\cH_K(G)} V^K = (\ind_{H_2}^G\chi_{2}')^K \otimes _{\cH_K}(\ind_{H_{1}}^G\chi_{1}')^K,
$$
where we consider the left ${\cH_K(G)}$-module $(\ind_{H_2}^G\chi_{2}')^K $ as a right one using the anti-involution $g \mapsto g^{-1}$.
The assumption implies that $(\ind_{H_i}^G\chi_{i}')^K $ are finitely generated over $\cH_K(G)$. By \Cref{thm:Ber}\eqref{2} this implies that they are also finitely generated over $\fz(\cH_K(G))$.
 Thus $M$ is also finitely generated over $\fz(\cH_K(G))$.
 The assertion follows now from \Cref{lem:comalg} in view of  \Cref{thm:Ber}\eqref{3}.
\end{proof}}

\section{Wavefront set of $\fz$-finite distributions and the proof of \Cref{thm:main}}\label{sec:WF}
\setcounter{lemma}{0}
In this section we assume that $F$ has characteristic zero.
%
\subsection{Preliminaries on wave front set}\label{subsec:WF}

\RamiD{
In this section we give an overview of the theory of the wave front set as developed  by D.~Heifetz \cite{Hef}, following L.~H\"ormander (see \cite[\S 8]{Hor}). }\RamiD{For simplicity we ignore here the difference between distributions and generalized functions.}
\begin{defn}\label{def:wf}$ $
\begin{enumerate}
\item
Let $V$ be a finite-dimensional vector space over $F$.
Let $f \in C^{\infty}(V^*)$ and $w_0 \in V^*$. We  say that $f$ \emph{vanishes asymptotically in the direction of} $w_0$
if there exists $\rho \in \Sc(V^*)$ with $\rho(w_0) \neq 0$ such that the function $\phi \in C^\infty(V^* \times F)$ defined by $\phi(w,\lambda):=f(\lambda w) \cdot \rho(w)$ is compactly supported.

\item
Let $U \subset V$ be an open set and $\xi \in \Sc^*(U)$. Let $x_0 \in U$ and $w_0 \in V^*$.
We say that $\xi$ is \emph{smooth at} $(x_0,w_0)$ if there exists a compactly supported non-negative function $\rho \in \Sc(V)$ with $\rho(x_0)\neq 0$ such that the Fourier transform $\cF^*(\rho \cdot \xi)$ vanishes asymptotically in the direction of
$w_0$.
\item
The complement in $T^*U$
of the set of smooth pairs $(x_0, w_0)$ of $\xi$ is called the
\emph{wave front set of} $\xi$ and denoted by $WF(\xi)$.

\item For a point $x\in U$ we denote $WF_x(\xi):=WF(\xi)\cap T^*_xU$.
\end{enumerate}
\end{defn}

\begin{remark}
Heifetz  defined  $WF_{\Lambda}(\xi )$ for any open subgroup $\Lambda$ of $F^{\times}$ of finite index. Our definition above is slightly different from the definition in \cite{Hef}. They relate by
\begin{equation*}
WF(\xi)-(U \times \{0\})= WF_{F^{\times}}(\xi).
\end{equation*}
\end{remark}

\RamiD{
\begin{prop}
[{see  \cite[Theorem 8.2.4]{Hor}  and \cite[Theorem 2.8]{Hef}}]\label{prop:SubPull}
\label{submrtion}
Let $U \subset F^m$ and $V \subset F^n$ be open subsets. Suppose that  $f: U \to V$ is an analytic submersion\footnote{\DimaB{\emph{i.e.} the differential of $f$ is surjective.}}. Then for any
$\xi \in \Sc^*(V)$, we have $$WF(f^*(\xi)) \subset f^*(WF(\xi)):=\left\{ (x,v) \in T^*U \vert \,  \exists w\in WF_{f(x)}(\xi) \text{ s.t. } d_{f(x)}^*f(w)=v \right \} .$$
\end{prop}

\begin{corollary} \label{iso}
Let $V, U \subset F^n$ be open subsets. Let $f: V \to U$ be an
analytic isomorphism. Then for any $\xi \in \Sc^*(V)$ we have
$WF(f^*(\xi)) = f^*(WF(\xi))$.
\end{corollary}}

\begin{corollary}\label{cor:bundle}
Let $X$ be an $F$-analytic manifold\footnote{In  the classical
sense of \cite{Ser} and not in the sense of rigid geometry or Berkovich geometry.}. We can define the wave front set of any distribution in
$\Sc^*(X)$, as a subset of the cotangent bundle $T^*X$.
\end{corollary}
We will need the following standard properties of the wavefront set.


\begin{lem}\label{lem:WFSmooth}
Let $X$  be an $F$-analytic manifold and
$\xi \in \Sc^*(X)$ be a distribution on it.
\begin{enumerate}
\item \label{it:smooth} Let $x\in X$. Assume that $WF_x(\xi)=\{0\}$. Then $\xi$ is smooth at $x$, i.e. there exists an analytic embedding $\phi:U \hookrightarrow  X$ from an open neighborhood $U$ of the origin in $F^n$ to $X$ such that
$\phi(0)=x$ and $\phi^*(\xi)$ coincides with a Haar measure.

\item \label{it:equiv} \cite[Theorem 4.1.5]{Aiz} Let an $F$-analytic group $H$ act analytically on $X$. Suppose that $\xi$ changes under the action of $H$ by some character of $H$. Then $$WF(\xi) \subset \{(x,v) \in T^*X|v(\fh x)=0 \},$$
where $\fh x$ denotes the image of the differential of the action map $h \mapsto hx$.
\end{enumerate}
\end{lem}

%

\subsection{Preliminaries on \DimaB{special} balls}\label{subsec:PrelFuzzy}

The notions of \DimaB{special} balls and admissible balls were defined in \cite{S} \DimaB{(under the name \emph{fuzzy balls})} for $\mathbf{G}=GL_n$. Here we generalize them to arbitrary reductive groups, using the standard theory of exponentiation.
\DimaA{
\begin{notn}
Let $\cO$ be the ring of integers in $F$. Fix  a uniformizer $\varpi \in \cO$ and
denote $q:=|\varpi|^{-1}$.
\end{notn}
}

We start with the following standard lemma on exponentiation.
\RamiF{\begin{lemma}\label{lem:exp}
There exists a lattice (i.e. a free finitely-generated $\cO$-submodule of full rank) $\cL \subset \fg$, a compact open subgroup $K\subset G$ and an analytic diffeomorphism $\exp:\cL \to K$ such that
\begin{enumerate}
\item \label{it:hom} For any $x \in \cL, \, \exp|_{\cO \cdot x} $ is a group homomorphism.
\item \label{it:norm} $\frac{d}{dt}\exp(tx)|_{t=0}=x.$
\item \label{it:BCH} For any $X\in  \varpi^{m}\cL, Y \in  \varpi^{n}\cL$ we have $$\exp^{-1}(\exp(X)\exp(Y))-X-Y \in \varpi^{m+n}\cL.$$
\end{enumerate}
\end{lemma}
For completeness we will indicate the proof of this lemma in \Cref{subsec:BCH}.
\begin{remark}
The conditions \eqref{it:hom} and \eqref{it:norm} define the map $\exp$ uniquely. \end{remark}
}

We fix such an $\cL$. Fix also an additive character $\psi$ of $F$ that is trivial on $\cO$ \DimaB{and non-trivial on $\varpi^{-1}\cO$}.

\begin{defn} $\,$
\begin{itemize}
\item For a vector space $V$ over $F$ and a lattice $\Lambda\subset V$ denote $\Lambda^{\bot}:=\{y\in V^* \, \vert \, \forall x\in \Lambda,\, \,  \langle x , y\rangle \in \cO\} \subset V^*$.

\RamiF{

\item For a set $B=a + \Lambda \subset \fg^*$ define a subset $K_B:=exp(\Lambda^{\bot}) \subset G$. Define also a function $\eta_B$ of $K_B$ by $\eta_B(exp(x))=\psi(\langle a , x \rangle)$. Note that  $K_B$ and $\eta_B$ depend only on the set $B$ and not on its presentation as $a + \Lambda$.

\item An \emph{admissible ball} is a set $B \subset \fg^*$ of the form $a + \Lambda,$ where $\Lambda\supset \cL^{\bot}$ is a lattice such that
$K_B$ is a group and $\eta_B$ is its character.

Define $e_B \in \cH(G)$ to be the measure $\eta_B e_{K_B}$, where $e_{K_B}$ is the normalized Haar measure on $K_B$.

}

\item An admissible ball $B$ is called \emph{nilpotent} if it intersects the nilpotent cone $\cN\subset \fg^*$.

\item For an element $x \in \fg^*$ we define $|x|:=\min\{|\alp| \, \vert x \in \alp \cL^{\bot}, \alp  \in     F\}$.

\item A \DimaB{special} ball of radius $r\geq 1$ is a set $B \subset \fg^*$ of the form $c + \alp \cL^{\bot}$, where $\alp \in F, \, c \in \fg^*, |\alp|=r$ and either $|c| =r^2$ or $|c\DimaA{\varpi}|=r^2$. \RamiC{It is easy to see that any \DimaB{special} ball is an admissible ball.}

\item For $Y\in \fg^*$ we denote by $B(Y)$ the unique \DimaB{special} ball containing $Y$ \DimaB{(see \Cref{lem:BallDisj})}.

\item Denote the set of all \DimaB{special} balls by $\mathfrak{F}$.

\item $\cL_n:=\DimaA{\varpi}^{n}\cL,\, K_n:=\exp(\cL_n)$ for $n \ge 0$.

\end{itemize}
\end{defn}

In \Cref{sec:balls} we give more details about admissible and \DimaB{special} balls and prove the following fundamental statements.

\begin{thm}\label{thm:fuzzy decomp}
Let $(\pi,V)$ be a smooth representation. Then $\{\pi(e_B)\}_{B\in \mathfrak{F}}$ form a full family of mutually orthogonal projectors, i.e.
\begin{enumerate}
\item \label{thm:fuzzy decomp:orthproj} for any $B,C \in \mathfrak{F}$ we have $$   e_{B}e_{C}= \begin{cases}e_B & B=C, \\
0 & B\neq C. \\
\end{cases}$$

\item \label{thm:fuzzy decomp:full} $$ \displaystyle V = \bigoplus_{B\in \mathfrak{F}} \pi(e_B)V.$$

\end{enumerate}

\end{thm}

\begin{thm}\label{thm:FuzzyFin}
For any finitely generated smooth representation $\pi$ there exist only finitely many non-nilpotent \DimaB{special} balls $B$ such that $\pi(e_B)\neq 0$.
\end{thm}

\begin{lemma}\label{lem:FouFuzz}
Let $B$ be an admissible ball and let $1_B\in \Sc(\fg^*)$ denote the characteristic function of $B$. Let $\cF(1_B)$  denote the Fourier transform of $1_B$\DimaB{, considered as a measure on $\fg$}. Then $\cF(1_B)=\exp^{*}(e_B )$.
\end{lemma}

\subsection{Proof of \Cref{thm:main}}\label{sec:PfWF}

We will need some preparations.
\begin{prop}[\RamiC{see \Cref{subsec:PfEigenFinSupp}}]\label{prop:EigenFinSupp}
Let $\xi \in \cS^*(G)$ be a $\fz$-finite distribution. Then there exists a compact subset $D\subset \g^*$  such that for any  non-nilpotent \DimaB{special} ball $B\subset \g^*\setminus D$ we have $e_{B}*\xi= 0$.
\end{prop}


The following is a straightforward computation.
\begin{lemma}
Let $B:=a+ \alp \RamiE{\cL^{\bot}}$ be an admissible ball\EitanA{, with $|\alpha|^{2}>|a|$}. Let $\DimaA{S}$ be the set of all \DimaB{special} balls contained in $B$. Then $e_B=\sum_{C \in \DimaA{S}}e_C$.
\end{lemma}
\RamiC{The last 2 statements give us the following corollary.}
\begin{cor}\label{cor:Bxi0}
Let $\xi \in \cS^*(G)$ be a $\fz$-finite distribution. Then there exists a compact subset $D\subset \g^*$  s.t. for any non-nilpotent admissible ball of the form  $B:=a+ \alp\RamiE{\cL^{\bot}}\subset \g^*\setminus D$\EitanA{, with $|\alpha|^{2}>|a|$,}
we have $e_{B}*\xi= 0$.
\end{cor}

\begin{prop}[\RamiC{See \Cref{subsec:pfFouExp}}]\label{prop:FouExp}
Let \EitanA{$n,l>0$ and let} $B=a+  \RamiE{\varpi^{-n}\cL^{\bot}}$ be an admissible ball. Assume that $|a|=\RamiE{q^{n+l}}$. Then for any $\xi \in \cS^*(\exp(\DimaA{\varpi}^l\cL))$ we have $$\exp^*(e_B *\xi)=\exp^*(e_B)*\exp^*(\xi)$$ \end{prop}

\begin{proof}[Proof of \Cref{thm:main}]
Since any shift of $\xi$ is also $\fz$-finite, we can assume that $x$ is the unit element $1\in G$. Thus it is enough to show that $WF_0(exp^*(\xi))\subset \cN.$

\RamiE{Let $Y \in \fg^*$ be non-nilpotent. Then there exists $m$ such that  for  all big enough $\alp \in F$ the set $\alp (Y+\DimaA{\varpi}^m\cL^{\bot})$  is a non-nilpotent admissible ball. Let $B:=Y+\DimaA{\varpi}^m\cL^{\bot}$.
There exists $l$ such that $\varpi^{l-m}Y\in \cL^{\bot}$.
Let $\phi$ be the characteristic function of $\DimaA{\varpi}^{l}\cL$ and $\mu$ be the characteristic function of $K_{l}:=exp(\DimaA{\varpi}^{l}\cL)$. Let $\zeta:=\mu\xi$ and $\eta:=\phi exp^*(\xi)=exp^*(\zeta)$. We have to show that for all big enough $\alp \in F$ we have $\cF(\eta)|_{\alp B}=0$.
By  \Cref{cor:Bxi0}, for all big enough $\alp \in F$ we have $$\RamiC{e_{\alp B} * \zeta =0}.$$}
By \Cref{prop:FouExp}  for all big enough $\alp \in F$ we have: $$\RamiC{exp^{*}(e_{\alp B} * \zeta )}=exp^{*}(e_{\alp B} )*\eta =\cF^{-1}(\cF(\exp^{*}(e_{\alp B} ))\cF(\eta)).$$
\Cref{lem:FouFuzz} implies now that $\cF(\eta)|_{\alp B}=0$ for all big enough $\alp \in F$.
\end{proof}

\subsection{Proof of \Cref{prop:EigenFinSupp}}\label{subsec:PfEigenFinSupp}

%

\begin{proof}[Proof of \Cref{prop:EigenFinSupp}]
Let $\pi:=\Sc(G)*\xi*\Sc(G)$. By \Cref{cor:adm}, $\pi$ is a \DimaB{finite length} representation of $G\times G$ and thus, by \Cref{thm:FuzzyFin}, there exists a finite set $X$ of \DimaB{special} balls of $G\times G$ such that $\pi(e_{\cB})=0$ for a non-nilpotent $\cB \notin X$. Let $D$ be the union of the projections of the balls in $X$ to the first coordinate. It is easy to see that for any non-nilpotent \DimaB{special} ball $B\subset\g^* \setminus D$ and any \DimaB{special} ball $C, \, B\times C \notin X$ and thus for any   $f \in \cS(G)$ we have
$$e_{B}*\xi*f*e_{C}=
e_{B}*e_{B}*\xi*f*e_{C}=\pi(e_{B\times C})(e_{B}*\xi*f)=0$$
By \Cref{thm:fuzzy decomp}, $$e_{B}*\xi*f= \sum_{C,C'\in \mathfrak{F}} e_{C'}*e_{B}*\xi*f*e_C= \sum_{C\in \mathfrak{F}} e_{B}*\xi*f*e_C,$$
where the sum goes over all \DimaB{special} balls in $\g^*$.
This implies $e_{B}*\xi*f=0$. Since this holds for any $f \in \cS(G), \, e_{B}*\xi$ vanishes.
\end{proof}

\subsection{Proof of \Cref{prop:FouExp}}\label{subsec:pfFouExp}

$\,$

%
\RamiF{ From standard properties  of the exponential map (see \Cref{lem:exp}\eqref{it:BCH}) we obtain the following Corollary.}

\begin{cor}\label{cor:exp}$\,$
\begin{enumerate}[(i)]
\item \label{it:exp}
For any natural number $n$ and any $a \in \cL$ we have
$$ \exp(a+\cL_n)=\exp(a)\exp(\cL_n)=\exp(a)K_{n}$$
\item \label{it:expL}
Let $e_{K_0}$ be the Haar probability measure on $K_0$. Then $\exp^*(e_{K_0})$ is the Haar probability measure on $\cL$.

\item \label{it:expbar}
Let $n$ and $l$ be natural numbers. By \eqref{it:exp} we can define $\overline{\exp}:\cL_0/\cL_{n+l}\to K_0/K_{n+l}$. Let $\alp,\beta$ be measures on $K_0/K_{n+l}$ such that $\alp$ is supported on $K_l/K_{n+l}$ and $\beta$ is supported on $K_n/K_{n+l}$. Then $$\overline{\exp}^*(\alp *\beta)=\overline{\exp}^*(\alp)*\overline{\exp}^*(\beta)$$
\end{enumerate}
\end{cor}

\begin{proof}[Proof of \Cref{prop:FouExp}]$\,$
\begin{enumerate}[Step 1.]
\item \label{FouExp:1} Proof for the case $l=0$.\\
In this case for any $\DimaB{b}\in \cL_0$ we have $$(e_{B} * \xi)|_{\exp(\DimaB{b})K_n}=(e_{K_n} * \xi)|_{\exp(\DimaB{b})K_n}=\left(\int_{\exp(\DimaB{b})K_n}\xi\right)(\#K_0/K_n) e_{K_0}|_{\exp(\DimaB{b})K_n}.$$
Also,
$$(\exp^*(e_B)*\exp^*(\xi))|_{\DimaB{b}+\cL_n}= \left(\int_{\DimaB{b}+\cL_n}\exp^*(\xi)\right)(\#\cL_0/\cL_n) \mu_{\cL_0}|_{\DimaB{b}+\cL_n},$$
where $\mu_{\cL_0}$ is the Haar probability measure on $\cL_0$.
By \Cref{cor:exp}\eqref{it:exp}, $$\exp^{-1}({\exp(\DimaB{b})K_n})=\DimaB{b}+\cL_n \text{ and } \int_{\exp(\DimaB{b})K_n}\xi=\int_{\DimaB{b}+\cL_n}\exp^*(\xi).$$
Thus, by \Cref{cor:exp}\eqref{it:expL}, $$exp^*(e_{K_0}|_{\exp(\DimaB{b})K_n})=\mu_{\cL_0}|_{\DimaB{b}+\cL_n}.$$
We get
\begin{multline*}
\exp^*(e_B*\xi)|_{(\DimaB{b}+\cL_n)}=\exp^*((e_B*\xi)|_{(\exp(\DimaB{b})K_n)})=\\
=\exp^*\left(\left(\int_{\exp(\DimaB{b})K_n}\xi\right)(\#K_0/K_n) e_{K_0}|_{\exp(\DimaB{b})K_n}\right)=\\=\left(\int_{\exp(\DimaB{b})K_n}\xi\right)(\#K_0/K_n)\exp^*(e_{K_0}|_{\exp(\DimaB{b})K_n})= \\ =\left(\int_{\exp(\DimaB{b})K_n}\xi\right)(\#K_0/K_n)\mu_{\cL_0}|_{\DimaB{b}+\cL_n}= \left(\int_{\DimaB{b}+\cL_n}\exp^*(\xi)\right)(\#\cL_0/\cL_n)\mu_{\cL_0}|_{\DimaB{b}+\cL_n}=\\=(\exp^*(e_B)*\exp^*(\xi))|_{\DimaB{b}+\cL_n}.
\end{multline*}

\item Proof for the general case.\\ Denote by $p_{\cL}$ and $p_K$ the natural projections $\cL_0\to \cL_0/\cL_{n+l}$ and $K_0\to K_0/K_{n+l}$. There exist measures \RamiC{$\beta$ and $\alp$ on $ K_0/K_{n+l}$ such that $e_{K_{n+l}}*\xi = p_K^*(\beta)$ and $e_B=p_K^*(\alp)$. Clearly $\supp (\beta)\subset K_l/K_{n+l}$ and $\supp (\alp)\subset K_n/K_{n+l}$.
We have
\begin{equation}\label{FouExp:=1}
\exp^*(e_B*\xi)=\exp^*(e_B*e_{K_{n+l}}*\xi)=\exp^*(p_K^*(\alp)*p_K^*(\beta))=
\exp^*(p_K^*(\alp*\beta)).
\end{equation}

From the commutative diagram
\begin{equation}\label{diag}\xymatrix{\parbox{20pt}{$\cL_0$}\ar@{->}^{\exp}[r]\ar@{->}^{p_{\cL}}[d]&
\parbox{20pt}{$K_0$}\ar@{->}^{p_K}[d]\\
\parbox{40pt}{$\cL_0/\cL_{n+l}$}\ar@{->}^{\overline{\exp}}[r] &
\parbox{40pt}{$K_0/K_{n+l}$}}
\end{equation}
 we have
\begin{equation}\label{FouExp:=2}
\exp^*(p_K^*(\alp*\beta))=p_L^*(\overline{\exp}^*(\alp*\beta)).
\end{equation}

By \Cref{cor:exp}\eqref{it:expbar}
we have
\begin{equation}\label{FouExp:=3}
p_L^*(\overline{\exp}^*(\alp*\beta))= p_L^*(\overline{\exp}^*(\alp) * \overline{\exp}^*(\beta)) =
p_L^*(\overline{\exp}^*(\alp))* p_L^*(\overline{\exp}^*(\beta)).
\end{equation}

Applying the diagram \eqref{diag} again we get
\begin{equation}\label{FouExp:=4}
p_L^*(\overline{\exp}^*(\alp))* p_L^*(\overline{\exp}^*(\beta))=\exp^*(p_K^*(\alp))* \exp^*(p_K^*(\beta))=\exp^*(e_B)*\exp^*(e_{K_{n+l}}*\xi).
\end{equation}

Applying Step \ref{FouExp:1} twice we have
\begin{multline}\label{FouExp:=5}
\exp^*(e_B)*\exp^*(e_{K_{n+l}}*\xi)
=\exp^*(e_B)*\exp^*(e_{K_{n+l}})*\exp^*(\xi)=\\=\exp^*(e_{K_{n+l}})*\exp^*(e_B)*\exp^*(\xi)=\exp^*(e_{K_{n+l}}*e_B)*\exp^*(\xi)=\exp^*(e_B)*\exp^*(\xi).
\end{multline}
Combining (\ref{FouExp:=1},\ref{FouExp:=2}-\ref{FouExp:=5}) we get $\exp^*(e_B*\xi)=\exp^*(e_B)*\exp^*(\xi)$.
}
\end{enumerate}
\end{proof}

\subsection{Regularity of invariant $\fz$-finite distributions at cotoric elements and proof of \Cref{thm:reg}}\label{sec:reg}

\setcounter{lemma}{0}

In this section we prove a generalization of \Cref{thm:reg}. We will need the following notion.

\begin{defn}
Let $\mathbf{H_1},\mathbf{H_2} < \mathbf{G}$ be algebraic subgroups. We say that an element $g\in G$ is $H_1 \times H_2$-\emph{cotoric} if the conormal space to $H_1gH_2$ at $g$ intersects trivially the nilpotent cone of $\fg^*$.
\end{defn}

\begin{lem}\label{lem:SymSmooth}
\DimaB{Let $\mathbf{H}$ be an open subgroup the group of fixed points of an involution  $\theta$ of $\mathbf{G}$.}
Let $g\in G$ be an $\mathbf{H} \times \mathbf{H}$-regular semisimple element, i.e. an element such that the double coset $\mathbf{H} g \mathbf{H}$ is closed and of maximal dimension. Then $g$ is $\mathbf{H} \times \mathbf{H}$-cotoric.

In particular, the set of cotoric elements contains an open dense subset of $\mathbf{H} \times \mathbf{H}$.
\end{lem}

\begin{proof}
 Let $\sigma$ be the anti-involution given by $\sigma(g):=\theta(g^{-1})$. Let $(\mathbf{H}\times \mathbf{H})_g$ be the stabilizer of $g$ with respect to the two-sided action of $ \mathbf{H} \times \mathbf{H}$, and $N_{\mathbf{H}g \mathbf{H},g}^{\mathbf{G}}$ be the normal space to the double coset $\mathbf{H}g \mathbf{H}$ at $g$ in $\mathbf{G}$. Since $g$ is $\mathbf{H} \times \mathbf{H}$-regular semisimple, the Luna slice theorem (see e.g. \cite[Theorem 5.4]{Dre}) implies that $(\mathbf{H}\times \mathbf{H})_g$ acts trivially on $N_{\mathbf{H}g \mathbf{H},g}^{\mathbf{G}}$.

Let $x=g\sigma(g)$. \DimaB{By \cite[Proposition 7.2.1(ii)]{AG_HC}, the pair consisting of the group  $(\mathbf{H}\times \mathbf{H})_g$ and its action on $N_{\mathbf{H}g \mathbf{H},g}^{\mathbf{G}}$ is isomorphic to the pair consisting of the centralizer $\mathbf{H}_x$ and   its adjoint action on the centralizer $\mathfrak{g}_x^{\sigma}$ of $x$ in the space $\mathfrak{g}^{\sigma}$ of fixed points of $\sigma$ in $\fg$.  Since $g$ is $\mathbf{H}\times\mathbf{H}$-semisimple,  \cite[Proposition 7.2.1(i)]{AG_HC}  shows that $x$ is a semisimple element of $\mathbf{G}$. Thus $\mathbf{G}_x$ is a reductive group.}

Now, assume that $x$ is not cotoric. Then, using a non-degenerate $\theta$-invariant and $\mathbf{G}$-invariant quadratic form on $\fg$ \DimaB{(see e.g. \cite[Lemma 7.1.9]{AG_HC})}, we can find a nilpotent element $\alpha \in \mathfrak{g}_x^{\sigma}$. Using the Jacobson-Morozov theorem for symmetric pairs (see e.g. \cite[Lemma 7.1.11]{AG_HC}), for some $t\neq 1\in F$ we can find an element $h\in H_x$ such that $ad(h)(\alpha)=t\alpha$. This contradicts the fact that $H_x$ acts trivially $\mathfrak{g}_x^{\sigma}$.
\end{proof}

In view of \Cref{lem:WFSmooth}, \Cref{thm:main} gives us the following corollary.

\begin{cor}\label{cor:sm}
Let $\mathbf{H_1},\mathbf{H_2} < \mathbf{G}$ be algebraic subgroups. Let $\chi_i$ be characters of $H_i$, and let $\xi$ be an $(H_1\times H_2,\chi_1\times \chi_2)$-equivariant $\DimaA{\fz}$-finite distribution. Let $x\in G$ be an $H_1 \times H_2$-\emph{cotoric} element. Then $\xi$ is smooth at $x$.
\end{cor}

In view of \Cref{lem:SymSmooth} this corollary implies \Cref{thm:reg}.


\appendix
\section{Fuzzy balls (joint with Alexander Kemarsky) }\label{sec:balls}
\setcounter{lemma}{0}

In this section we prove the statements on admissible balls and \DimaB{special} balls formulated in \Cref{subsec:PrelFuzzy}. \RamiC{We} follow \cite[\S 4 and \S 5.1]{S}. Throughout the section we assume that $F$ has characteristic zero.
\RamiF{
\subsection{The exponential map and proof of \Cref{lem:exp}}\label{subsec:BCH}

It is enough to prove \Cref{lem:exp} for $\mathbf{G}=\GL_n$. Consider the power series $$\mathcal{E}xp(X):=\sum_{k=0}^{\infty}a_k X^k:=\sum_{k=0}^{\infty} X^k/k!  \text{ and } \mathcal{L}og(X):=\sum_{k=0}^{\infty}b_k (X-1)^k:=\sum_{k=1}^{\infty} (-1)^{k-1}(X-1)^k/k,$$
where  $X\in \Mat_{n\times n}(F)$. The Baker-Campbell-Hausdorff formula is the following equality of power series
\begin{multline} \label{=BCH}
\mathcal{L}og(\mathcal{E}xp(X+Y))-X-Y=\\=\sum_{n=1}^{\infty}\sum_{|i|+|j|= n}\left(c_{ij} ad^{i_1}_X \circ ad^{j_1}_Y \cdots ad^{i_k}_X \circ ad^{j_k}_Y(X)  +d_{ij} ad^{i_1}_Y \circ ad^{j_1}_X \cdots ad^{i_k}_Y \circ ad^{j_k}_X(Y) \right),
\end{multline}
where $i=(i_1,\dots,i_{k})$ and $j=(i_1,\dots,i_{k})$ are multi-indices and $c_{ij},d_{ij}\in \bQ\subset F$ are certain constants. Let \EitanA{$\alp_n:=\max(|a_n|,|b_n|,\max_{|i|+|j|=n}(|c_{ij}|),\max_{|i|+|j|=n}(|d_{ij}|))$}. It is well known for some constant $C>1$ and all $n\geq 1$ we have $\alp_n\leq C^n$. Define $\cL:=\{X\in \Mat_{n\times n}(F)  \, \vert \, |X_{ij}|<C^{-1}\}$. It is easy to see that the power series $\mathcal{E}xp$ converge on $\cL$. We define $\exp$ to be the corresponding analytic map, and $K$ to be $\exp(\cL)$. Finally, it follows from \eqref{=BCH} that $(\cL,K,\exp)$ satisfy the requirements \eqref{it:hom}-\eqref{it:BCH}.
\proofend
}

\subsection{Proof of \Cref{thm:fuzzy decomp} }\label{subsec:ballsDecomp}
We start with the following easy lemma.

\begin{lemma}\label{lem:BallDisj} Let $\cB$ denote the collection of all \DimaB{special} balls. Then $\fg^*$ decomposes as a disjoint union
$$\gotG^* = \coprod_{B\in\cB} B.$$
\end{lemma}
\begin{proof}
Let $X \in \gotG^*$. If $|X| \le 1$, then $X \in \RamiE{\cL_1^\bot}$. If $|X| = q^m > 1$, then
$X \in \RamiE{X+\cL^{\bot}_{[\frac{m}{2}]}} $. Thus $\gotG^* = \bigcup B$. Let $\RamiE{B_1 = X+\cL_m^{\bot},\,
B_2 = Y+\cL_n^{\bot}}$ be \DimaB{special} balls and suppose that $z = X + l_1 = Y+l_2 \in B_1 \cap B_2$.
Then $|z| = |X| = |Y|$, so $m=n$. Let $Y+l' \in B_2$. We can rewrite this element as
$$Y+l' = Y+l_2 + l'-l_2 = z + (l'-l_2) \in X + \cL_m^{\perp} = B_1.$$
We have obtained $B_2 \subset B_1$ and clearly by the same argument applied to $B_1$, we obtain
$B_1 \subset B_2$. Therefore, $B_1 = B_2$.
\end{proof}
\RamiF{Let $(\pi,V)$ be a smooth representation.
The following lemma is straightforward.}
\begin{lemma}\label{lem:proj}
For an admissible ball $B$, the image of $\pi(e_B)$ consists of $(\eta_B^{-1},K_B)$-equivariant vectors in $V$, i.e.
$$\pi(e_B)V = \left\{ v \in V: \pi(k)v = \eta_B^{-1}(k) v \;\; \forall k\in K_B  \right\}.
$$ Moreover, $e_B$ is a projection, that is $e_B=e_B^2$.
\end{lemma}

\begin{lemma}\label{lem:orthproj}
Let $B_1, B_2$ be distinct \DimaB{special} balls. Then $e_{B_1} e_{B_2} = 0$
\end{lemma}
\begin{proof}
Suppose $B_1 \ne B_2$ and $e_{B_1}e_{B_2} \ne 0$. Then for any $a \in K :=
K_{B_1} \cap K_{B_2}$ we have $$\eta_{B_1}(a^{-1})e_{B_1} e_{B_2}=
ae_{B_1} e_{B_2} = e_{B_1}a e_{B_2} = \eta_{B_2}(a^{-1})
e_{B_1}e_{B_2}  $$
We get $\eta_{B_1}|_K = \eta_{B_2}|_K$. Now, if $K_{B_1} = K_{B_2}$, then
$B_1 = B_2$, a contradiction. Otherwise we can assume $K_{B_1} \subset K_{B_2}$,
but then the character $\eta_{B_1}$ is a restriction of $\eta_{B_2}$ from the bigger
group $K_{B_2}$, thus $B_1$ and $B_2$ intersect and thus by \Cref{lem:BallDisj} they coincide, which again is a contradiction.
\end{proof}

\begin{lemma}\label{lem:BigBallVan}
Let  $v \in V^{K_N}$ and $B$ be a \DimaB{special} ball with radius bigger
than $q^N$. Then $\pi(e_B)v = 0$.
\end{lemma}
\begin{proof}
$$\pi(e_B)v = \int_{K_B} \eta_B(k) \pi(k)v dk = \left(\int_{K_B} \eta_B(k) dk \right)v = 0.$$
\end{proof}

\begin{lemma}\label{lem:FuzzyNonVan}
For every $0 \neq v \in V$, there exists a \DimaB{special} ball $B$, such that
$\pi_B(v) \neq 0$.
\end{lemma}
\begin{proof}
Let $0 \ne v \in V$. If $v \in V^{K_0}$ then $v \in V(B)$ for $B = 0 +\cL^{\bot}$.\\
Suppose $v \not \in V^{K_0}$. Let $n \ge 1$ be the minimal $n$ such that
 $v \in V^{K_{2n}},
v \not \in V^{K_{2n-2}}$. Thus the group $A = K_n / K_{2n}$ acts on the finite dimensional
space $W$ generated in $V$ by the orbit $K_n v$. Note that $K_n/K_{2n} \simeq \cL_n/ \cL_{2n}$ and
by \RamiF{standard properties  of the exponential map (see \Cref{lem:exp}\eqref{it:BCH})}
the group $\cL_n/\cL_{2n}$ is commutative.
Thus, the group $A$ is a commutative finite group.
The space $W$ can be decomposed as a direct sum of  one-dimensional characters of $A$. For a character $\chi$ of $A$ and $w \in W$,
let $w(\chi) \in W$ be the projection of $w$ to the $\chi$-isotypic component of $W$.
Then
$v =\sum_{\chi \in A^{*}} v(\chi)$, and thus there exists a character $\chi$ of $A$ with $v(\chi)\neq 0$ and $\chi|_{K_{2n-2}/K_{2n}}\neq 1$. Lift $\chi$ to a character $\eta$ of $K_n$ and note that there exists a unique \DimaB{special} ball $B$ with $K_B=K_{n}$ and $\eta_B=\eta$. Then $\pi(B)v=v(\chi)\neq 0$.
\end{proof}

\begin{proof}[Proof of \Cref{thm:fuzzy decomp}]
Part \eqref{thm:fuzzy decomp:orthproj} follows from \Cref{lem:proj,lem:orthproj}. To prove part \eqref{thm:fuzzy decomp:full} take $0 \neq v \in V$, and let $w = \sum_{\DimaA{ B \in \mathfrak{F}}} \pi(e_B)v$. By \Cref{lem:BigBallVan} we know that the
sum is well-defined. By \Cref{lem:proj,lem:orthproj}, $\pi(e_B)(v-w) = 0$ for all \DimaB{special} balls $B$. By \Cref{lem:FuzzyNonVan}, it
follows that $v = w$.
\end{proof}

\subsection{Proof of \Cref{thm:FuzzyFin}}\label{subsec:ballsFin}

\begin{defn} A \emph{\DimaB{special} set} is a finite union of \DimaB{special} balls. For a \DimaB{special} set $T = \cup B_i$, denote
$e_T := \sum e_{B_i}$.
\end{defn}

\begin{lemma}\label {lem: fuzzy adjoint}
Let $T$ and $S$ be two \DimaB{special} sets in $\gotG^*$ and let $g\in G$. Let $(\pi,V)$ be a  smooth  representation of $G$. Suppose that $\pi(e_T) \pi(g) \pi(e_S) \ne 0$. Then
$ad(g)S \cap T \ne \emptyset$.
\end{lemma}
\begin{proof}
By linearity we reduce to the case where $$T=B=X + \cL_m^{\bot} \text{ and } S=B'=X' + \cL_{m'}^{\bot}$$ are \DimaB{special} balls. Let
$$K= K_B=\exp(\cL_m), \,K' = K_{B'}=\exp(\cL_{m'}),\,\eta = \eta_B, \eta' = \eta_{B'}$$  We first check that
$$\eta(a) = \eta'(g^{-1}ag) $$ for all $a \in K \cap gK'g^{-1}$. Indeed, let
$v \in V$ be such that $\pi(e_B) \pi(g) \pi(e_B')(v) \ne 0$. Then for all $a \in K \cap gK'g^{-1}$
we have
\begin{equation}\label{eq:1}
\eta(a)^{-1} \pi(e_B) \pi(g) \pi(e_{B'})(v) = \pi(a) \pi(e_B )\pi(g) \pi(e_{B'})(v) =
\pi(e_B) \pi(a) \pi(g) \pi(e_{B'})(v),
\end{equation}
since $\pi(e_B) \pi(a) = \pi(a) \pi(e_B)$ for $a \in K$. On the other hand,
\begin{equation}\label{eq:2}
\pi(e_B) \pi(a) \pi(g) \pi(e_{B'})(v) = \pi(e_B) \pi(g) \pi(g^{-1}ag) \pi(e_{B'})(v) =
\eta'^{-1}(g^{-1}ag) \pi(e_B)\pi(g) \pi(e_{B'})(v).
\end{equation}
Combining equations (\ref{eq:1}) and (\ref{eq:2}), we obtain $\eta(a)
= \eta'(g^{-1}ag)$ for every $a \in K \cap gK'g^{-1}$.
Therefore, $$ \psi_0\left(X(\log(a))\right) = \psi_0
\left(X'(Ad(g^{-1})\log(a))\right) = \psi_0\left( (Ad(g)X')(\log(a)) \right). $$ \\
We see that for $b \in \cL_m \cap Ad(g)\cL_{m'}$,
$$(Ad(g)X' - X)(b) \in \OO.$$ Thus,
$$ Ad(g)X' - X \in (\cL_m \cap Ad(g)\cL_{m'})^{\perp} = \cL_m^{\perp} +
Ad(g)\cL_{m'}^{\perp},$$
that is, there exist $u \in \cL_m^{\perp}$ and $v \in \cL_{m'}^{\perp}$ such that
$$ Ad(g)X' - X = u + Ad(g)v.$$
Hence $Ad(g)(X'-v) = X + u \in X +\cL_m^{\bot}=B$.
\end{proof}

The following Lemma due to Howe plays a central role in the proof of \Cref{thm:FuzzyFin}.
\begin{lemma}[{\cite[Lemma 12.2]{HCDBS}}]\label{lem:comp}
Let $S \subset \gotG$ be compact. There exists a compact subset $S_1$ such that $$Ad(G)S
\subset S_1 + \NN.$$
\end{lemma}

\begin{proof}[Proof of \Cref{thm:FuzzyFin}]
Suppose $V$ is generated by $v_1,...,v_n$ and for each $i$ pick all \DimaB{special} balls
$B_{ij}$ such that $\pi(e_{B_{ij}})(v_i) \ne 0$. Note that
by \Cref{{lem:BigBallVan}} for every $v \in V$, there are only finitely
many \DimaB{special} ball $B$ such that $\pi(e_B)v \ne 0$.  Let $S = \cup B_{ij}$.
By \Cref{thm:fuzzy decomp} $\pi(S)v_i=v_i$.
 Since $S$ is compact, \Cref{lem:comp} implies
$$Ad(G)S \subset \cL_{m}^{\bot} + \NN$$ for some large $m$.
Let $B$ be a \DimaB{special} ball such that $\pi(e_B) \ne 0.$  Let us show that there exists  $g \in G$ such
that $\pi(e_B) \pi(g) \pi(e_S) \ne 0$.
Indeed, suppose on the contrary that $\pi(e_B) \pi(g) \pi(e_S) = 0$ for every
$g \in G$. Let $v$ be such that $\pi(e_B)v \neq 0$ and write $$v = \sum_{j=1,1 \le i_j \le n }^{k} c_i \pi(g_i) v_{i_j}. $$
Then $$\pi(e_B)(v) = \sum_{j=1,1 \le i_j \le n }^{k} c_i \pi(e_B) \pi(g_i) v_{i_j}
= \sum_{j=1,1 \le i_j \le n }^{k} c_i \pi(e_B) \pi(g_i) \pi(e_S) v_{i_j} = 0, $$
and we obtain a contradiction!
By Lemma \ref{lem: fuzzy adjoint} $$Ad(g)S \cap B \ne \emptyset.$$
In particular $B \cap (\cL_{m}^{\bot} + \NN) \ne \emptyset.$ Suppose $B = X+\cL_n^\bot$ with $n \ge m$ and
let $Y \in B \cap (\cL_{m}^{\bot}+ \NN)$. Then $Y = X + l = l' + n$, $l \in \cL_m^{\perp}, l' \in \cL_n^{\perp}$, and
$n \in \NN$. In particular, $n = X + (l-l') \in X + \cL_n^{\perp}$, so $n \in B \cap \NN$. \\
We have obtained that every \DimaB{special} ball $B$ that acts on $V$ as non-zero and has big enough radius is
a nilpotent \DimaB{special} ball. Since the number of \DimaB{special} balls with a bounded radius is finite, we obtain that
all except of finitely many non-nilpotent balls act on $V$ as zero.
\end{proof}

\subsection{Proof of \Cref{lem:FouFuzz}}\label{subsec:PfFouFuzz}
We follow \cite[\S 5.1]{S}.
Assume that $ B = B(X,L)$.
Note that $\exp^*(e_B) = f \mu$, where $\mu$ is the Haar measure on $L$, normalized such
that $\mu(L)=1$ and $f$ is a function given by $f(y) =\psi_0(\langle y,X\rangle) 1_L(y)$.
Then $$\cF(\mu)(Z) = \int_{y \in L} \psi_0(\langle y,X+Z\rangle) d\mu(y).$$
The last integral is an integral of an additive character on an additive group. Such an integral is zero, unless
the character is trivial. In our case this means that the integral is zero, unless
$X+Z \in L^{\perp}$, which happens if and only if $-Z \in X + L^{\perp}$ and in that case the integral equals $1$.
Therefore,  $\cF(\mu)(Z) = 1_{X +L^{\perp}}(-Z)$. As $\cF \circ \cF = -Id$ (under the identification
$\gotG \simeq \gotG^{**}$), we get that $\cF(1_B) = \exp^*(e_B)$, as claimed.

\section{Finite Generation of Hecke Modules (by A. Aizenbud and D. Gourevitch)} \label{sec:FinGen}
In this section we prove \Rami{a stronger version of \Cref{thm:FinGen}. For its formulation we will need the following definition.
\begin{defn}$ $
Let $(G,(H,\chi))$ be a twisted pair, i.e. $H <G$ is a (closed) subgroup, and $\chi$ is its character.
\begin{enumerate}
\item Denote by $D_{G/H}$ the $G$-equivariant sheaf of smooth measures on $G/H$ and by $\Delta_{G/H}$ its fiber at $[1]\in G/H$, considered as a character of $H$.
Note that $\Delta_{G/H}=(\Delta_{G})|_H\cdot\Delta_{H}^{-1}=\Delta_{H}^{-1}$.
\item We define the \emph{dual} of the twisted pair $(G,(H,\chi))$ to be the pair $(G,(H,\DimaA{\hat \chi}))$, where $\DimaA{\hat \chi}=\Delta_{G/H}\chi^{-1}$. \RamiC{Note that $\DimaA{\hat{ \hat \chi}}=\chi.$}
\end{enumerate}
\end{defn}

The following theorem clearly implies \Cref{thm:FinGen}.
\begin{theorem}\label{thm:ExpFinGen}
Let $(G,(H,\chi))$ be a twisted pair.
Let $\mathbf{P}$ be a minimal parabolic  subgroup of $\mathbf{G}$ and $P=\mathbf{P}(F)$. Suppose that $H$ has finitely many orbits on $G/P$.
Suppose that for any
irreducible smooth representation $\rho$ of $G$ we have
\begin{equation} \label{eq:FinMultExp}
\dim\Hom_{H}(\rho|_{H}, \DimaA{\hat \chi}) < \infty .
\end{equation}
 Then for any open compact subgroup $K$ of $G$ the module $\ind_H^G(\chi)^K$ over the Hecke algebra $\cH_K(G)$ is finitely generated.

\end{theorem}
}
Let us now  give an overview of the argument.
In Lemma \ref{VKFinGen} we present a criterion, due to Bernstein, for the finite generation of spaces of $K$-invariants. The proof of the criterion uses the theory of Bernstein Center. \Dima{Using this criterion we \RamiC{introduce a notion}  of twisted pairs of finite type \RamiC{(see \Cref{def:Fin}\eqref{Fin:type})} which is equivalent to the \RamiC{local} finite generation of  $\ind_H^G(\chi)$.
Bernstein's} criterion is given in terms of all parabolic subgroups of $G$. This allows us to define an intermediate notion of finite cuspidal  type \RamiC{(see \Cref{def:Fin}\eqref{Fin:CuspType})}, which means that the criterion holds for the group $G$ as a parabolic subgroup of itself.

Then we introduce duality between twisted pairs. We prove that
\Dima{condition \eqref{eq:FinMultExp}
implies that }the dual  pair $(G,(H, \DimaA{\hat \chi}))$ is of finite cuspidal  type (see \Cref{subsec:CuspFinType}). We use a simple trick (\Cref{cor:CuspMultType}) to imply that the pair $(G,(H,\chi))$ is itself of finite cuspidal  type.

In order to analyze the condition of  Lemma \ref{VKFinGen} for all parabolic subgroups of $G$ we introduce the notion of a descendant of the pair $(G,(H,\chi))$ and prove that if all the descendants are of finite cuspidal  type then the pair itself is of finite  type (see \Cref{SecPfFinGen}).

Thus it remains to show that if the conditions of \Cref{thm:FinGen} hold for a twisted pair then they hold for all its descendants. This we do in \Cref{subsec:desc.cusp}, using a homological algebra argument.

%
%


\begin{rem}\label{rem:AAG}
The argument here is an adaptation of a similar argument in \cite{AAG} that dealt with the case of trivial $\chi$.
However, the argument in \cite{AAG} did not take into account the modular characters of various groups that appear along the way. As a result, it is not valid for non-unimodular $H$ and even for  unimodular $H$ it is not complete. The gap in the original  argument is filled \RamiC{mainly} by the proof of \Cref{cor:CuspMultType}.
\end{rem}
\subsection{Preliminaries} \lbl{subsec:prel}

%

\begin{notn}
 For a subgroup $H<G$ we denote by $\ind_H^G: \cM(H) \to \cM(G)$ the compactly supported induction functor and by $\Ind_H^G: \cM(H) \to \cM(G)$ the full  induction functor. \RamiC{For $\pi \in \cM(G)$ denote by $\widetilde{\pi}:=(\pi^*)^{\infty}$ the smooth contragredient representation. Note that for any character $\chi$ of $H$ we have $\widetilde{\ind_H^G(\chi)}=\Ind_H^G(\DimaA{\hat \chi})$.}
\end{notn}

\begin{defn}
Let $\mathbf{P}<\mathbf{G}$ be a parabolic subgroup with unipotent radical $\mathbf{U}$, and let $\mathbf{M}:=\mathbf{P}/\mathbf{U}$.
Such $\mathbf{M}$ is called a Levi subquotient of $G$.
Note that every representation of $M$ can be considered as a  representation of $P$ using the quotient morphism $P \twoheadrightarrow M$.
Define:
\begin{enumerate}
 \item The Jacquet functor $r^G_{M}:\cM(G) \to \cM(M)$ by $r^{G}_{M}(\pi):=(\EitanA{\Delta_{G/P}^{\frac{1}{2}}}\cdot\pi|_{P})_{U}$.
 Note that $r^G_{M}$ is defined for any representation of $P$.
 \item The parabolic induction functor $i^G_{M}:\cM(M) \to \cM(G)$ by $i^G_{M}(\tau):=\ind_{P}^{G}(\EitanA{\Delta_{G/P}^{-\frac{1}{2}}}\tau)$.
\end{enumerate}
Note that $i^{G}_{M}$ is right adjoint to $r^{G}_{M}$.
A representation $\pi$ of $G$ is called cuspidal if $r^{G}_{M}(\pi)=0$ for any Levi subquotient $\mathbf{M}$ of $\mathbf{G}$.
\end{defn}

It is well-known that $i^G_M$ and $r^G_M$ are exact functors.

\begin{definition}
A smooth representation $V$ of $G$ is
called \emph{compact} if for any $v \in V$ and $\xi \in
\widetilde{V}$ the matrix coefficient function defined by
$m_{v,\xi}(g):= \xi(gv)$ is a compactly supported function on $G$.
\end{definition}

\begin{theorem}[Bernstein-Zelevinsky]\lbl{CompProj}
Any compact representation of $G$ is
a projective object in the category $\cM(G)$.
\end{theorem}

\begin{definition}$\,$
\begin{enumerate}
 \item  Denote by $G^1$ the preimage in $G$ of the maximal compact
subgroup of $G/[G,G]$.
\item Denote by $Z(G)$ the center of $G$ and denote $G_0:=G^1Z(G)$.
\item \RamiC{We call } a complex character of $G$  \emph{unramified} if it is trivial
on $G^1$. We denote the set of all unramified
characters by $\Psi_G$. Note that $G/G^1$ is a lattice and therefore we can identify $\Psi_G$ with $(\bC^{\times})^n$. This defines a structure of algebraic variety on $\Psi_G$.
\item For any smooth representation $\rho$ of $G$ we denote
$\Psi(\rho):= \ind_{G^1}^G(\rho|_{G^1})$. Note that $\Psi(\rho) \simeq  \rho \otimes \cO(\Psi_G),$
where $G$ acts only on the first factor, but this action depends on the  second factor.
This identification gives a structure of $\cO(\Psi_G)$-module on $\Psi(\rho)$.
\end{enumerate}
\end{definition}


\begin{theorem}[Harish-Chandra]\lbl{CuspComp}
Let $V$ be a cuspidal representation
of $G$. Then $V|_{G^1}$ is a compact representation of $G^1$.
\end{theorem}

\begin{corollary} \lbl{CuspProj}
Let $\rho$ be a cuspidal
representation
of $G$. Then\\
(i) $\rho|_{G^1}$ is a projective object in the category
$\cM(G^1)$.\\
(ii) $\Psi(\rho)$ is a projective object in the category
$\cM(G)$.
\end{corollary}
\begin{proof}
(i) \RamiC{Follows from \Cref{CuspComp,CompProj}}.\\
(ii) note that $$Hom_G(\Psi(\rho),\pi) \cong
Hom_{G/G_1}(\cO(\Psi_M),Hom_{G^1}(\rho,\pi)),$$
for any representation $\pi$. Therefore the functor $\pi \mapsto
Hom_G(\Psi(\rho),\pi)$ is a composition of two exact functors and
hence is exact.
\end{proof}


We will use Bernstein's second adjointness theorem.

\begin{thm}[{\cite{BerSec} or \cite[Theorem 3]{Bus}}]\label{thm:2adj}
\RamiC{Let $\mathbf{P}\subset \mathbf{G}$ be a parabolic subgroup and let $\overline{\mathbf{P}}$ be an opposite parabolic subgroup. Let $\mathbf{M}$ be the Levi quotient of $\mathbf{P}$} and let $\overline{r}^G_{M}:\cM(G) \to \cM(M)$ denote the Jacquet functor defined using $\overline{P}$.
Then $\overline{r}^G_{M}$ is right adjoint to $i^G_{M}$. In particular, $i^G_{M}$ maps projective objects to projective ones and hence for any irreducible cuspidal
representation $\rho$
of $M$, 
$i^G_{M}(\Psi(\rho))$ is a projective object of $\cM(G)$.
\end{thm}

We now present a criterion, due to Bernstein, for \RamiC{local} finite generation. 

\begin{lemma}[{\cite[Lemma 2.1.10]{AAG}}] \lbl{VKFinGen}
Let $V\in \cM(G)$. Suppose that for any
parabolic $P<G$ and any irreducible cuspidal representation $\rho$
of $M$ (where $M$ denotes the reductive quotient
of $P$),
$ \Hom_{G}(i^G_{M}(\Psi(\rho)),V)$ is a finitely generated
module over $\cO(\Psi_M)$. Then $V^K$ is a finitely generated
module over  $\fz(\mathcal{H}_K(G))$, for any compact open
subgroup $K<G$.
\end{lemma}

The theory of Bernstein center gives us the following Lemma:

\begin{lem}\label{lem:Noeth}
Let $V$ be a smooth finitely generated representation of $G$. Let $W \subset V$ be a subrepresentation.
Then $W$ is finitely generated.
\end{lem}
\begin{proof}
Let $v_1, \dots, v_n$ be the generators of $V$. By \Cref{thm:Ber}\eqref{it:split} we can choose a splitting subgroup $K \subset G$ s.t. $v_i \in V^K$. Then $V\in \cM_K(G)$ and $V^K$ is finitely generated   over $\cH_K(G)$. Hence  $W\in \cM_K(G)$ and thus it is enough to show that $W^K$ is finitely  generated over $\cH_K(G)$.
By \Cref{thm:Ber}\eqref{2} $\cH_K(G)$ is finite over its center $\fz(\mathcal{H}_K(G))$. So  $V^K$ is  finitely generated   over $Z(\cH_K(G))$. From \Cref{thm:Ber}\eqref{3} it follows that  $\fz(\cH_K(G))$ is Notherian, and thus $W^K$ is finitely generated  generated over $\fz(\cH_K(G))$.
\end{proof}

\subsection{Finite multiplicity and duality of twisted pairs}\label{subsec:FinDual}
Let $(G,(H,\chi))$ be a twisted pair.
\begin{defn}\label{def:Fin}
 We say that the pair  $(G,(H,\chi))$ \begin{enumerate}
\item \label{Fin:mult} has \emph{finite multiplicities} \RamiC{(resp. \emph{finite cuspidal multiplicities})} if for any irreducible \RamiC{(resp. cuspidal irreducible)} representation $\pi$ of $G$, $$\dim\Hom(\ind_{H}^G(\chi),\pi)<\infty.$$
\item \label{Fin:type} has \emph{finite type} if  for any
parabolic \RamiC{$\mathbf{P}<\mathbf{G}$} and any irreducible cuspidal representation $\rho$
of $M$ (where $\mathbf{M}$ denotes the reductive quotient of $\mathbf{P}$),
$$ \Hom_{G}(i^G_{M}(\Psi(\rho)),\ind_{H}^G(\chi))$$ is a finitely generated
module over $\cO(\Psi_M)$.

\item \label{Fin:CuspType} has \emph{finite cuspidal  type} if  for any irreducible cuspidal representation $\rho$
of $G$,
$ \Hom_{G}(\Psi(\rho),\ind_{H}^G(\chi))$ is a finitely generated
module over $\cO(\Psi_M)$.
\item is $F$-\emph{spherical}\footnote{
If $char F=0$ and $G$ is quasi-split over $F$ then $(G,H)$ is an $F$-spherical pair if and only if
it is a spherical pair of algebraic groups.
} if for any
parabolic subgroup $\mathbf{P} \subset \mathbf{G}$, there is a finite number of
double cosets in $P \setminus G / H$.

\end{enumerate}

\end{defn}


The following lemma helps to connect multiplicities to duality.

\begin{lem}\label{lem:tilde}
Let $\pi,\rho \in \cM(G)$ and assume that $\rho$ has \DimaB{finite length}. Then the natural morphism $\Hom(\pi,\rho) \to \Hom(\widetilde{\rho},\widetilde{\pi}) $ is an isomorphism.
\end{lem}

\begin{proof}
By \Cref{thm:Ber} we can choose a splitting subgroup $K$ such that $\rho^K$ generates $\rho$. Then
$$
\Hom(\pi,\rho) \cong \Hom_{\cH_K(G)}(\pi^K,\rho^K)\cong \{v \in (\pi^K)^*\otimes \rho^K \, \vert \, \forall a \in \cH_K(G), \, (a\otimes 1 - 1 \otimes a)v = 0\}.$$
Here we consider the standard action of $\cH_K(G)^{opposite}\otimes \cH_K(G)$ on acts on $(\pi^K)^*\otimes \rho^K$. Also
\begin{multline*}\Hom(\widetilde{\rho},\widetilde{\pi})\cong\Hom_{\cH_K}((\rho^{*})^{K},(\pi^*)^K) \cong  \Hom_{\cH_K}((\rho^K)^*,(\pi^{K})^*)\cong\rho^K \otimes (\pi^*)^K \cong \\ \cong\{v \in \rho^K \otimes (\pi^K)^*  \, \vert \, \forall a \in \cH_K(G), \, (a\otimes 1 - 1 \otimes a)v = 0\}
\end{multline*}
This easily implies the assertion.
\end{proof}

Using Frobenius reciprocity we obtain the following corollary.

\begin{cor}\label{lem:DualFinMult}
Let $\pi$ be a \DimaB{smooth} representation of $G$ of \DimaB{finite length}. Then
$$\dim\Hom_G(\ind_{H}^G(\chi),\pi)=\dim\Hom_H(\tilde \pi, \DimaA{\hat \chi})$$
\end{cor}

\RamiA{\Cref{lem:admFin} follows from \Cref{lem:DualFinMult} and the next lemma.

\begin{lem}\label{lem:FinTypeFinMult}
If $(G,(H,\chi))$ has  \EitanA{finite (cuspidal)} type then it has  \EitanA{finite (cuspidal)} multiplicities.
\end{lem}
\begin{proof}
 Let $\pi$ be a irreducible (cuspidal) representation of $G$. By \Cref{thm:Ber} we can choose a splitting compact open subgroup $K<G$ s.t. $\pi^K \neq 0$. Then $$\Hom_{G}(\ind_H^G(\DimaA{\hat \chi}),\pi) = \Hom_{\cH_K(G)}((\ind_H^G(\DimaA{\hat \chi}))^K,\pi^K).$$ Since $(\ind_H^G(\DimaA{\hat \chi}))^K$ is finitely generated, this implies that $\dim \Hom_{G}(\ind_H^G(\DimaA{\hat \chi}),\pi)<\infty$.
\end{proof}
}

%

In view of \Cref{VKFinGen,lem:DualFinMult}, \Dima{\Cref{thm:ExpFinGen}} is equivalent to  the following one.

\begin{thm}\label{thm:EqFinGen}
If $(G,(H,\chi))$ is $F$-spherical and has finite multiplicity then it has finite type.
\end{thm}

\subsection{Descent Of
Finite Multiplicity} \lbl{subsec:desc.cusp}

%
%
%
%

\begin{notation}
Let $(G,(H,\chi))$ be a twisted pair.
 Let \RamiC{$\mathbf{P}<\mathbf{G}$ be
a parabolic subgroup and $\mathbf{M}$ be its Levi quotient. Let $\mathbf{\overline{P}}<\mathbf{G}$ be
a parabolic subgroup opposite to $\mathbf{P}$ and $\mathbf{\overline{U}}$ be its unipotent radical.} Let $X:=G/H$.
\Rami{Let $\cF$ be the natural $G$-equivariant sheaf on $X$ such that the stalk at $[1]$ coincides with $\chi$ as a representation of $H$.}

Let $x \in X$. It is easy to see that there exists a geometric quotient $A^{x} = \overline{U}\backslash\overline{P}x$. Denote by $\cF^{x}$  the natural $M$-equivariant sheaf on $A^{x}$ such that $\overline{r}_M^G(\Sc(\overline{P}x, \cF|_{\overline{P}x}))=\Sc(A^{x},\cF^{x})$.
Suppose that $\cF^{x}\neq 0$. Let $y$ be the image of $x$ in $A^{x}$. We denote its stabilizer in $M$ by $H^x_M$, and we consider the fiber $(\cF^{x})_y$ as a character of $H^x_M$, and denote it by $\chi_M^x$.

We say that the twisted pair $(M, (H_M^x,\chi_M^x))$ is a \emph{$P$-descendent} of the twisted pair $(G,(H,\chi))$.  We will say that descendants $(M, (H_M^x,\chi_M^x))$ and $(M, (H_M^{x'},\chi_M^{x'}))$ are equivalent if $x$ and $x'$ belong to the same $P$-orbit.
\end{notation}

The following \RamiC{version of the Bernstein-Zelevinsky geometric} lemma follows from the exactness of $\overline{r}_M^G$.

\begin{lem}\label{lem:geo}
\RamiC{Let $\mathbf{P}<\mathbf{G}$ be
a parabolic subgroup and $\mathbf{M}$ be its Levi quotient.} Let $(G,(H,\chi))$ be an $F$-spherical twisted pair. Then $ \ind_H^G(\chi)$ has a finite filtration such that $$Gr(\overline{r}_M^G( \ind_H^G(\chi)))\simeq\bigoplus_{i} \ind_{H_i}^G(\chi_i), $$
where $(M,(H_i,\chi_i))$ ranges over all the $P$-descendants of $(G,(H,\chi))$ up to equivalence.
\end{lem}

The goal of this subsection is to prove the following lemma.

\begin{lemma} \lbl{FinMultCuspDesc}
Let $(G,(H,\chi))$ be an $F$-spherical pair of finite multiplicity. Let $P<G$ be a parabolic
subgroup and $M$ be its Levi quotient.  Let $(M,(H',\chi'))$ be a descendant of $(G,(H,\chi))$.  Then $(M,(H',\chi'))$ is  an $F$-spherical pair of finite cuspidal multiplicity.
\end{lemma}

\begin{remark}
One can easily show that the converse statement to this lemma is also true. Namely, if all the descendants of the pair  have finite cuspidal multiplicity then the pair has finite multiplicity.  This also implies that in this case all the descendants have finite multiplicity.
However we will not use these facts.
\end{remark}

We will need the following lemmas.

\begin{lemma} \lbl{LinAlg}
Let $M$ be an l-group and $V,W$ be smooth representations of $M$ such that
that $\dim \Hom(V,W) < \infty$.
Let $0=F^0V \subset ... \subset F^{n-1}V \subset F^nV=V$ be a
finite filtration of $V$ by subrepresentations. Suppose that for
any $i$, either $$\dim \Hom(F^iV/F^{i-1}V,W) = \infty$$ or $$\text{both }\dim\Hom(F^iV/F^{i-1}V,W)< \infty \text{ and }\dim \Ext^{1}(F^iV/F^{i-1}V,W) <
\infty.$$ Then $\dim\Hom(F^iV/F^{i-1}V,W) < \infty$ for any $i$.
\end{lemma}
\begin{proof}
We prove by a decreasing induction on $i$ that $\dim\Hom(F^iV,W)<\infty$, and, therefore, $\dim\Hom(F^iV/F^{i-1}V,W)<\infty$ and by the conditions of the lemma $$\dim \Ext^{1}(F^iV/F^{i-1}V,W)<\infty.$$ Consider the short exact sequence
$$0 \to F^{i-1}V \to F^iV \to F^iV/F^{i-1}V \to 0,$$
and the corresponding long exact sequence
$$...\ot \Ext^1(F^iV/F^{i-1}V,W)  \ot \Hom(F^{i-1}V,W) \ot \Hom(F^iV,W) \ot \Hom(F^iV/F^{i-1}V,W) \ot 0.$$
In this sequence $\dim \Ext^{1}(F^iV/F^{i-1}V,W) < \infty$ and $\dim \Hom
(F^iV,W) < \infty$, and hence $\dim \Hom(F^{i-1}V,W) < \infty$.
\end{proof}

%

\begin{lemma}
\lbl{FinDimH1H0}
Let $(G,(H,\chi))$ be a twisted pair.
Let $\rho$ be an irreducible cuspidal representation of
$G$.  Suppose that $\dim \Hom(\ind_H^G(\chi), \rho)< \infty$. Then $\dim \Ext^1(\ind_H^G(\chi), \rho)< \infty$.
\end{lemma}

For the proof we will need the following straightforward lemma.

\begin{lemma} \lbl{AFG}
Let $L$ be a lattice. 
Let $V$ be a linear
space. Let $L$ act on $V$ by a character. Then
$$\oH_1(L,V) =  \oH_0(L,V) \otimes_{\bC}(L \otimes_{\bZ}\bC).$$
\end{lemma}

\begin{proof}[Proof of \Cref{FinDimH1H0}]
By \Cref{lem:tilde} $$\Ext^i(\ind_H^G(\chi)\RamiC{,\rho}) \cong \Ext^i(\widetilde{\rho},\Ind_H^G(\DimaA{\hat \chi})).$$ By Frobenius reciprocity $$\Ext^i(\widetilde{\rho},\Ind_H^G(\DimaA{\hat \chi}))\cong \Ext^i_H(\widetilde{\rho},\DimaA{\hat \chi}).$$ Let $I:=H\cap G^1$ and $J:=H\cap G_0$.  Note that
$$ \Ext^i_H(\widetilde{\rho},\DimaA{\hat \chi} )\cong \Ext^i_H(\widetilde{\rho}\otimes \DimaA{\hat \chi}^{-1}, \bC) \cong \Ext^i_{H/I}((\widetilde{\rho}\otimes \DimaA{\hat \chi}^{-1})_{I}, \bC),$$
where the last isomorphism follows from \Cref{CuspProj}. Now\RamiC{, since $H/J$ is finite,}
$$ \Ext^i_{H/I}((\widetilde{\rho}\otimes \DimaA{\hat \chi}^{-1})_{I}, \bC) \cong \Hom _{H/J}(\oH_i(J/I,(\widetilde{\rho}\otimes\DimaA{\hat \chi}^{-1})_{I}),\bC),$$
which implies the assertion by \Cref{AFG}.
\end{proof}

%

Now we are ready to prove Lemma \ref{FinMultCuspDesc}.
\begin{proof}[Proof of Lemma \ref{FinMultCuspDesc}]
Clearly $(M,H')$ is $F$-spherical. It remains to prove that
\begin{equation}\label{eq:CuspDesMult}
\dim \Hom(\ind_{H'}^{M}(\chi'),\tau)<\infty,
\end{equation} for any irreducible cuspidal representation $\tau$ of $M$.

Since $(G,(H,\chi))$ has finite multiplicity, we have  $\dim \Hom(\ind_H^G(\chi),\pi)<\infty$ for any irreducible $\pi\in \cM(G)$. Thus for any irreducible $\tau \in \cM(M)$ we have
$$\dim \Hom_G(\ind_H^G(\chi),\overline{i}_{M}^G(\tau)))<\infty.$$
Thus $$\dim \Hom_M(\overline{r}^G_{M}(\ind_H^G(\chi)),\tau)<\infty.$$ By \Cref{lem:geo},
there exists a filtration on $\overline{r}^G_{M}(\ind_H^G(\chi))$ such that $Gr^i(\overline{r}^G_{M}(\ind_H^G(\chi)))=\ind_{H_i}^G(\chi_i)$ where $(M,(H_i,\chi_i))$ ranges over all the descendants of $(G,(H,\chi))$ up to equivalence. In particular we can assume that for some $i_0$, we have $H_{i_0}=H',\, \chi_{i_0}=\chi'$.
By \Cref{FinDimH1H0}, this filtration satisfies the conditions of \Cref{LinAlg} and thus \eqref{eq:CuspDesMult} holds.
\end{proof}

\subsection{Finite cuspidal type} \label{subsec:CuspFinType}
Let us now prove the following cuspidal analogue of theorem \ref{thm:EqFinGen}
\begin{thm}\label{thm:EqCuspFinGen}
If $(G,(H,\DimaA{\hat \chi}))$  has \EitanA{finite cuspidal} multiplicity, then $(G,(H,\chi))$ has \EitanA{finite cuspidal} type.
\end{thm}

We will need several lemmas.

%
%

\begin{lemma}\label{lem:CharExt}
Let $A$ be a locally compact group and $B$ be a closed subgroup. Suppose that $A=BZ(A)$. Then any character of $B$ can be lifted to $A$.
\end{lemma}
\begin{proof}
Taking quotient by the kernel of the character we reduce to the case of abelian $A$. In this case the statement is \cite[Theorem 5]{Dix}.
\end{proof}

\begin{lemma}
\label{G1}
Let $(G,H)$ be an $F$-spherical pair, and denote $\widetilde{H}=HZ(G)\cap G^1$. Let $\chi$ be a character of $H$. Suppose that for any smooth
(respectively cuspidal) irreducible representation $\rho$ of
$G$ we have $$\dim\Hom_{H}(\rho|_{H}, \chi) < \infty$$ Then
for any smooth (respectively cuspidal) irreducible representation
$\rho$ of $G$ and for every character $\psi$ of $\widetilde{H}$ whose restriction to $H\cap G^1$ coincides with $\chi$, we have
$$\dim\Hom_{\widetilde{H}}(\rho|_{\widetilde{H}}, \psi) < \infty.$$
\end{lemma}
\begin{proof}
Let $\rho$ be a smooth (respectively cuspidal) irreducible
representation of $G$.  Using Lemma \ref{lem:CharExt} extend $\chi$ to a character $\chi'$ of $HZ(G)$. Let $\phi$ be a character of $\widetilde{H}$ whose restriction to $H\cap G^1$ is trivial.
 We have to show that $$\dim \Hom_{\widetilde{H}}\left (\rho|_{\widetilde{H}}, \phi\chi' \right) < \infty.$$ We have
 \[
\Hom_{\widetilde{H}}\left (\rho|_{\widetilde{H}},\phi\chi' \right ) =
\Hom_{H Z(G)\cap G_0} \left (\rho|_{(H Z(G))\cap G_0}, \Ind_{\widetilde{H}}^{H Z(G)\cap G_0}\phi\chi'\right).
\]
Since $$HZ(G)\cap G_0=\widetilde{H}Z(G)\cap G_0=\widetilde{H}Z(G),$$
the subspace of $\Ind_{\widetilde{H}}^{(H Z(G))\cap G_0}\phi\chi'$ that transforms under $Z(G)$ according to the central character of $\rho$ is at most one dimensional. If this subspace is $0$, then the lemma is clear. Otherwise, denote it by $\tau$. Since $H\cap G^1$ is normal in $HZ(G)$, we get that the restriction of $\Ind_{\widetilde{H}}^{(H Z(G))\cap G_0}\phi$ to $H\cap G^1$ is trivial, and thus that $\tau|_{H\cap G^1}=\chi'_{H\cap G^1}$. Hence
\begin{multline*}  \Hom_{\widetilde{H}}\left (\rho|_{\widetilde{H}}, \phi\chi' \right )=
\Hom_{(H Z(G))\cap G_0} \left (\rho|_{(H Z(G))\cap G_0}, \tau\right)=\\=\Hom_{H\cap G_0} \left (\rho|_{H \cap G_0}, \tau|_{H\cap G_0}\right)
=\Hom_{H}\left(\rho|_{H},\Ind_{H\cap G_0}^{H}\tau|_{H\cap G_0}\right).
\end{multline*}
Since $H / H\cap G_0$ is finite and abelian, we have
$$\Ind_{H\cap G_0}^{H}(\tau|_{H\cap G_0})=\chi\left(\bigoplus_{i=1}^N \chi_i\right)$$
where $\chi_i$ are  characters of $H$, s.t.  $(\chi_i)|_{H\cap G^1}=1$. \RamiC{By \Cref{lem:CharExt}} the characters $\chi_{i}$ can be extended to  characters of $G$, because $H/(H\cap G^1)$ is a sub-lattice of $G/G^1$. Denoting the extensions by $\Theta_{i}$, we get that
\[
\Hom_{H}\left(\rho|_{H},\chi\chi_i\right)=\Hom_{H}\left((\rho\otimes\Theta_{i}^{-1})|_{H},\chi\right),
\]
but $\rho\otimes\Theta_{i}^{-1}$ is again smooth (respectively cuspidal) irreducible representation of $G$, so this last space is finite-dimensional.
\end{proof}

 \begin{lemma} \lbl{CA}
 Let $A$ be a commutative unital Noetherian algebra without zero divisors and let $K$ be its field of fractions. Let $K^\bN$ be the space of all sequences of elements of $K$. Let $V$ be a finite dimensional subspace of $K^\bN$ and let $M:=V \cap A^\bN$. Then $M$ is finitely generated.
\end{lemma}

\begin{proof} Since $A$
does not have zero divisors, $M$ injects into $K^\bN$. There is a number $n$ such that the projection of $V$ to $K^{\{1,\ldots n\}}$ is injective. Therefore, $M$ injects into $A^{\{1,\ldots n\}}$, and, since $A$ is Noetherian, $M$ is finitely generated.
\end{proof}



\begin{lemma}\label{fg}
Let $L$ be an $l$-group, and let $L' \subset L$ be an open normal subgroup of $L$ such that $L/L'$  is a lattice. Let $\rho$ be a smooth representation of $L$ of countable dimension. Suppose that for any character $\chi$ of $L$ whose restriction to $L'$ is trivial we have $$\dim\Hom_{L}(\rho, \chi) < \infty.$$
Consider $\Hom_{L}(\rho, \Sc(L/L'))$ as a representation of $L$, where $L$ acts by $((hf)(x))([y])=(f(x))([y h])$. Then this representation is finitely generated.
\end{lemma}

\begin{proof} By assumption, the action of $L$ on $\Hom_{L}(\rho,\Sc(L/L'))$ factors through $L/L'$. Since $L/L'$ is discrete, $\Sc(L/L')$ is the group algebra $\bC[L/L']$. We want to show that $\Hom_{L}(\rho,\bC[L/L'])$ is a finitely generated module over $\bC[L/L']$.

Let $\bC(L/L')$ be the fraction field of $\bC[L/L']$. Choosing a countable basis for the vector space of $\rho$, we can identify any $\bC$-linear map from $\rho$ to $\bC[L/L']$ with an element of $\bC[L/L']^\bN$. Moreover, the condition that the map intertwines the action of $L/L'$ translates into a collection of linear equations that the tuple in $\bC[L/L']^\bN$ should satisfy. Hence, $\Hom_{L'}(\rho,\bC[L/L'])$ is the intersection of the $\bC(L/L')$-vector space $\Hom_{L}(\rho,\bC(L/L'))$ and $\bC[L/L']^\bN$. By Lemma \ref{CA}, it suffices to prove that $\Hom_{L}(\rho,\bC(L/L'))$ is finite dimensional over $\bC(L/L')$.

Since 
$L$ is separable, and $\rho$ is smooth and of countable dimension,
there are only countably many linear equations defining $\Hom_{L}(\rho,\bC(L/L'))$; denote them by $\phi_1,\phi_2,\ldots\in\left(\bC(L/L')^\bN\right)^*$. Choose a countable subfield $K\subset\bC$ that contains all the coefficients of the elements of $\bC(L/L')$ that appear in any of the $\phi_i$'s. If we define $W$ as the
$K(L/L')$-linear subspace of
$K(L/L')^\bN$
defined by the $\phi_i$'s, then $\Hom_{L}(\rho,\bC(L/L'))=W\otimes_{K(L/L')} \bC(L/L')$, so $\dim_{\bC(L/L')}\Hom_{L}(\rho,\bC(L/L'))=\dim_{K(L/L')}W$.

Since $L/L'$ is a lattice generated by, say, $g_1,\ldots,g_n$, we get that $K(L/L')=K(t_1^{\pm 1},\ldots,t_n^{\pm 1})$ $=K(t_1,\ldots,t_n)$. Choosing elements $\pi_1,\ldots,\pi_n\in\bC$ such that $tr.deg_K(K(\pi_1,\ldots,\pi_n))=n$, we get an injection $\iota$ of
$K(L/L')$ into $\bC$. As before, we get that if we denote the $\bC$-vector subspace of $\bC^\bN$ cut by the equations $\iota(\phi_i)$ by $U$, then $\dim_{K(L/L')}W=\dim_{\bC}U$. However, $U$ is isomorphic to $\Hom_{L}(\rho,\chi)$, where $\chi$ is the character of $L/L'$ such that $\chi(g_i)=\pi_i$. By assumption, this last vector space is finite dimensional.
\end{proof}
Now we are ready to prove Theorem \ref{thm:EqCuspFinGen}.
\begin{proof}[Proof of Theorem \ref{thm:EqCuspFinGen}]
Let $\rho$ be an irreducible cuspidal representation of $G$. By \Cref{lem:DualFinMult} we know that $\dim \Hom_{H} (\rho, \chi) < \infty$. We need to show that $
\Hom_{G}(\Psi(\rho),\ind_{H}^{G}\chi)
$ is finitely generated over $\cO(\Psi_G)$.
We have
$$
\Hom_{G}(\Psi(\rho),\ind_{H}^{G}\chi) =
\Hom_{G^1}(\rho,\ind_{H}^{G}\chi).$$
Here we consider the space $\Phi:=\Hom_{G^1}(\rho,\ind_{H}^{G}\chi)$ with
the natural action of $G$. Note that $G^1$ acts trivially and
hence this action gives rise to an action of $G/G^1$, which gives
the $\cO(\Psi_G)$ - module structure. Let $\Theta:=\Hom_{G^1}(\rho,\Ind_{HZ(G)}^{G}\ind_{H}^{HZ(G)}\chi)$.
Clearly $\Phi \subset \Theta$. Thus, by \Cref{lem:Noeth}, it is enough to show that $\Theta$ is finitely generated over $G$.

Denote $H':=H \cap
G^1$ and $H'':= (H Z(G))\cap G^1$.
Consider the subspace $$V:=\Hom_{G^1}(\rho, \Ind_{H''}^{G^1}(\ind_{H'}^{H''}(\chi|_{H'})))  \subset \Theta.$$  It generates
$\Theta$ as a representation of $G$,
and therefore also as an $\cO(\Psi_G)$ - module. Note that $V$ is $H$-invariant. Therefore it is enough to show that $V$ is finitely generated over $H.$


By Frobenius reciprocity we have $V \cong \Hom_{H''}(\rho,\ind_{H'}^{H''}(\chi|_{H'}))$.

 By Lemma \ref{lem:CharExt} $\chi$ can be extended to a character $\chi'$ of $H Z(G)$. Thus $$\ind_{H'}^{H''}(\chi|_{H'}) \cong \chi'|_{H''}\Sc(H''/H').$$

Let $\rho':=\chi'|_{H}^{-1}\rho|_{H}$. Then
$$ V \cong \Hom_{H''}(\rho',\Sc(H''/H')).$$
Under this isomorphism, the action $\Pi$ of $H$ on $V$ is given by $$((\Pi(h)(f))(v))([k])=f(\rho'(h^{-1})v)([h^{-1}kh]),$$
where $h\in H, \, f\in \Hom_{H''}(\rho',\Sc(H''/H')), \, v \in \rho', \, k \in H'', \, [k]=kH'\in H''/H'.$

Let $\Xi$ be the action of $H''$ on $\Hom_{H''}(\rho',\Sc(H''/H'))$ as described in  \Cref{fg}, i.e. $$((\Xi(h)(f))(v))([k])=f(v)([kh]).$$

Let us show that $\Hom_{H''}(\rho',\Sc(H''/H'))$ is finitely generated w.r.t. the action $\Xi$.  By \Cref{fg} it is enough to show that
\begin{equation}\label{=H''Fin}
\dim \Hom_{H''}(\rho', \theta) < \infty
\end{equation}
 for any character $\theta$ of $H''$ with trivial restriction to $H'$. Note that $\Hom_{H''}(\rho', \theta)\cong \Hom_{H''}(\rho, \chi' \theta)$. Thus \eqref{=H''Fin} follows from the hypothesis $\dim \Hom_{H}(\rho, \chi) < \infty$ in view of
 \Cref{G1} and we have shown that $\Hom_{H''}(\rho',\Sc(H''/H'))$ is finitely generated w.r.t. the action $\Xi$.

Now it is enough to show that for any $h \in H''$ there exist an $h' \in H$ and a scalar $\alpha$ s.t. $$\Xi(h)=\alpha \Pi(h').$$ In order to show this let us decompose $h$ to a product $h=zh'$ where $h' \in H$ and $z\in Z(G)$. Now
\begin{multline*}
((\Xi(h)(f))(v))([k])=f(v)([kh])=f(h^{-1}v)([h^{-1}kh])=f(h^{'-1}z^{-1}v)([h^{'-1}kh'])=\\=
\alpha f(h'^{-1}v)([h^{'-1}kh'])= \alpha((\Pi(h')(f))(v))([k]),
\end{multline*}
where $\alpha$ is the scalar with which
$z^{-1}$ acts on $\rho'$.

Thus $V$ is finitely generated over $H$, thus $\Phi$ and $\Theta$ are finitely generated over $G$ and $\Hom_{G}(\Psi(\rho),\ind_{H}^{G}\chi)$ is finitely generated over $\cO(\Psi_G)$.
\end{proof}

\begin{corollary}\label{cor:CuspMultType}
If $(G,(H, \chi))$  has   finite cuspidal multiplicity, then $(G,(H,\chi))$ has finite cuspidal type.
\end{corollary}
\begin{proof}
Assume \EitanA{that} $(G,(H, \chi))$  has   finite cuspidal multiplicity. By \Cref{thm:EqCuspFinGen} the twisted pair $(G,(H, \DimaA{\hat \chi}))$ has finite cuspidal type.
\RamiA{ By
\Cref{lem:FinTypeFinMult} the pair  $(G,(H, \DimaA{\hat \chi}))$  has   finite cuspidal multiplicity. Applying \Cref{thm:EqCuspFinGen} again we obtain that  $(G,(H,\chi))$ has finite cuspidal type.}
\end{proof}
\subsection{Proof of \Cref{thm:EqFinGen}} \lbl{SecPfFinGen} $\, $

Let $\mathbf{P} < \mathbf{G}$ be a parabolic subgroup and $\mathbf{M}$  be the Levi quotient of $\mathbf{P}$. Let $\rho$ be a cuspidal representation of $M$. We have to  show that $\Hom(i_M^G(\Psi(\rho)),\ind_H^G(\chi))$ is finitely generated over $\cO(\Psi_M)$. By second adjointness theorem (\Cref{thm:2adj}), we have $$\Hom(i_M^G(\Psi(\rho)),\ind_H^G(\chi))=\Hom((\Psi(\rho)),\bar r_M^G(\ind_H^G(\chi))).$$
 By \Cref{lem:geo} the representation $\bar r_M^G(\ind_H^G(\chi))$ has a filtration s.t. $$Gr_i(\bar r_M^G(\ind_H^G(\chi)))=\ind_{H_i}^M(\chi_i)$$ where $(M,(H_i,\chi_i))$ are the descendants of $(G,(H,\chi))$.
Since $i_M^G(\Psi(\rho))$ is  a projective object (\Cref{thm:2adj}), this gives us filtration on $\Hom((\Psi(\rho)),\bar r_M^G(\ind_H^G(\chi)))$ with $$Gr_i\Hom((\Psi(\rho)),\bar r_M^G(\ind_H^G(\chi)))=\Hom((\Psi(\rho)),\ind_{H_i}^M(\chi_i)).$$
  So it remains to show that $(M,(H_i,\chi_i))$  are of  finite cuspidal type. This follows from \Cref{FinMultCuspDesc,cor:CuspMultType}.

\end{document}